\documentclass[10pt]{ijnam}
\def\currentvolume{1}
\def\currentissue{1}
\def\currentyear{2004}
\def\month{March}
\pagespan{1}{18}
\copyrightinfo{2024}{} 
\usepackage{amsmath}
\usepackage[mathscr]{eucal}
\usepackage{multirow}
\usepackage{booktabs}
\usepackage{arydshln}
\usepackage{amssymb}
\newtheorem{lemma}{Lemma}[section]
\newtheorem{theorem}{Theorem}[section]
\newtheorem{remark}{Remark}[section]
\newtheorem{example}{Example}[section]
\newtheorem{case}{\bf Case}

\def\3bar{{|\!|\!|}}

\begin{document}

\title[Analysis of  any order Runge-Kutta Spectral Volume Schemes]
{Analysis of  any order Runge-Kutta Spectral Volume Schemes  for 1D Hyperbolic Equations}

\author[Wei P]{Ping Wei$^{1}$}
\address{
 $^1$School of Computer Science and Engineering, Sun Yat-sen University, Guangzhou 510006, China
}
\email{weip7@mail2.sysu.edu.cn}

\author[Zou QS$^*$]{Qingsong Zou$^{2*}$}
\address{
 $^{2*}$Corresponding author. School of Computer Science and Engineering, Sun Yat-sen University, Guangzhou 510006, China
}
\email{mcszqs@mail.sysu.edu.cn}




\thanks{The research was supported in part by the National Key R$\&$D Program of China (2022ZD0117805), by the National Natural Science Foundation of China under grant 12071496, and by the Guangdong Province grant  2023A1515012079.}


\abstract{In this paper, we analyze any-order Runge-Kutta spectral volume schemes (RKSV(s,k)) for solving the one-dimensional scalar hyperbolic equation. 
 The RKSV(s,k) was constructed by using the $s$-th explicit Runge-Kutta method in time-discretization which has {\it strong-stability-preserving} (SSP) property, and by letting a piecewise $k-$th degree($k\geq 1 $ is an arbitrary integer) polynomial satisfy the local conservation law in each control volume designed by subdividing the underlying mesh with $k$ Gauss-Legendre points (LSV) or right-Radau points (RRSV).
For the RKSV(s,k), we would like to establish a general framework which use the matrix transferring process technique for analyzing the stability and the convergence property. 
The framework for stability is evolved based on the energy equation, while the framework for error estimate is evolved based on the error equation. And the evolution process is represented by matrices.
After the evolution is completed, three key indicative pieces of information are obtained: the termination factor $\zeta$, the indicator factor $\rho$, and the final evolved matrix. 
We prove that for the RKSV(s,k), the {\it  stability } holds and the $L_2$ norm error estimate is $\mathcal{O}(h^{k+1}+\tau^s)$, provided that the CFL condition is satisfied. Our theoretical findings have been justified by several numerical experiments.}


\maketitle

\section{Introduction}\label{sec1}	
In this paper, we propose an analysis framework to obtain the L$_2$-norm stability and the convergence properties of the explicit any order Runge-Kutta  spectral volume schemes(RKSV(s,k)),  when solving the hyperbolic equation
\begin{align}\label{linearEQ}
    \partial_tu+\partial_xu=0,\quad  (x,t)\in[a,b]\times(0,T].
\end{align}
The equation with initial  condition
\begin{align}
    u(x,0)&=u_0(x),\quad x\in[a,b], \label{linearEQ_boundary_x}
\end{align}
where  $u_{0}$ is a known function, and with the homogeneous boundary condition $u(a,t)=0, t\in [0,T]$ or with the periodic boundary  condition $u(a,t)=u(b,t), t\in [0,T]$. 


For temporal discretization, we choose the Runge-Kutta method.
As is well known, there are many ways to select coefficients for constructing the discrete-time derivatives of an $s$-th order Runge-Kutta method. Different choices of Runge-Kutta methods for different parameters exhibit different performances, with some able to maintain {\it total-variation-diminishing}(TVD) properties\cite{cockburn1998introduction}, some able to maintain {\it strong-stability-preserving} (SSP) properties \cite{gottlieb2001strong} and so on. These properties are highly applicable to the numerical construction of conservation equations. 
In this paper, we consider the type of time-marching has been later termed SSP, which is widely applied in the analysis of nonlinear stability including the TVD property  and the positivity-preserving property \cite{zhang2010positivity} for nonlinear conservation laws.
This class of methods has a unified formula for deriving and calculating the coefficients, enabling us to establish a unified framework for analysis of any order.

For spatial discretization, we employ the spectral volume(SV) method, which was first proposed in 2002 by  Wang \cite{zhang_zhimin_SV_2002_1} for solving hyperbolic equations. 
The basic idea of the SV method is to ensure that a (discontinuous) piecewise high-order polynomial satisfies sub-element-level conservation laws.  
 Due to its numerous advantages, such as local conservation properties, geometric and mesh flexibility, high-order accuracy in smooth regions, and high resolution in discontinuous areas \cite{sun_wang_2004_comparison}, the SV method has been successfully applied to various partial differential equations (PDEs).
 Applications include Burgers' equation \cite{kannan2012high}, the shallow water equation \cite{Cozzolino_sv_2012, choi2004spectral}, the Euler equation \cite{wang_SV_2004, wang_SV_2004_four, Abeele_sv_2009_2, Abeele_SD_2007, liu2017high}, electromagnetic field equations \cite{liu2006spectral}, and the Navier-Stokes equation \cite{Sun_sv_2012, haga2009rans, sun2006spectral}.

In \cite{zhang2005analysis}, Wang and Shu classify the spectral volume (SV) method as a special Petrov-Galerkin method. While widely applied, the mathematical theory of the SV method, particularly its stability and convergence properties, has been less developed. Early theoretical studies focused on the stability of low-order schemes over uniform grids. Abeele et al. discovered that subdividing points significantly influence stability and noted that third and fourth-order SV schemes based on Gauss-Lobatto points are weakly unstable \cite{Abeele_sv_1D_2007, Abeele_sv_2D_2007}. Zhang and Shu proved the stability of 1st, 2nd, and 3rd order SV schemes on 1D uniform meshes using Fourier analysis \cite{zhang2005analysis}. Recently, Cao and Zou \cite{zou_cao_2021} developed a uniform framework for analyzing the stability, convergence, and superconvergence of SV schemes over nonuniform meshes for 1D scalar hyperbolic equations, extending their analysis to 1D and 2D hyperbolic equations with degenerate variable coefficients \cite{xu_cao_zou_2022, Cao-Zhang-Zou_2023}. Their main approach involves constructing a novel trial-to-test space mapping, allowing the SV method to be reformulated as a special Galerkin method.

In \cite{wei_zou_2023}, Wei et al. analyzed the stability and convergence properties of two fully discrete schemes: the forward Euler spectral volume scheme and the second-order Runge-Kutta spectral volume scheme. They utilized a unified formula to derive and calculate the coefficients, demonstrating that under various Courant-Friedrichs-Lewy (CFL) conditions, both schemes exhibit optimal convergence rates. The analysis involves checking the eigenvalues of the amplification matrix and ensuring they lie within the unit circle for stability. This paper builds on these findings by examining the stability and convergence properties of any-order Runge-Kutta spectral volume (RKSV) methods for hyperbolic equations, further extending the theoretical framework and providing a comprehensive analysis of stability criteria and convergence orders for these fully discrete schemes.

 For the fully discrete RKSV for linear coefficient hyperbolic equations, we aim to establish a general framework for analyzing stability and optimal order convergence properties.
It is worth noting that the SV method is essentially a Petrov-Galerkin method rather than a pure Galerkin one \cite{zhang2005analysis}.
In recent years, there has been extensive literature on the properties of fully discrete Runge-Kutta discontinuous Galerkin methods (RKDG) schemes, analyzing the stability and optimal order convergence properties for the explicit third-order and fourth-order RKDG methods \cite{zhang_shu_2010_error_3order,zhang2011third,zhang2014error_rk3,xu2020error_rk4}.
In \cite{xu_zhang_2019}, Xu et al. propose a unified framework to investigate the $L_2$-norm stability of the explicit RKDG using the matrix transferring process technique.
The main approach in the stability analysis is to establish a robust energy equation that clearly reflects the evolution of the 
$L_2$-norm of the numerical solution and explicitly shows the stability mechanism hidden in the fully discrete scheme.
 Inspired by this, we follow the original ideas in \cite{xu_zhang_2019,zhang2004error,zhang_shu_2010_error_3order} and make the following important developments in this paper.

In this paper, we consider a type of time-marching method later termed \textit{strong-stability-preserving} (SSP) \cite{gottlieb2001strong}, which finds wide application in the analysis of nonlinear stability, including properties such as total-variation-diminishing in the mean (TVDM) \cite{cockburn1998introduction} and positivity preservation for nonlinear conservation laws \cite{zhang2010positivity}.
We employ two spatial discretization techniques: one using Gauss-Legendre points (LSV) and the other using right-Radau points (RRSV).
 Next we analysis the stability and  optimal order convergence property of RKSV(s,k) for  hyperbolic problems. Here, $s$ represents the number of temporal stages, and $k$ represents the degree of the polynomial in the spatial domain.
 
 First on the stability, we point out that in the analysis of the RKSV, based on a from-trial-to-test-space mapping, a special norm $\3bar\cdot\3bar$ , which is equivalent to the $L_2$ norm, has been introduced \cite{zou_cao_2021}. This special norm can be rewritten SV methods as a Petrov-Galerkin methods. 
 Building on this foundation, we proceed with the stability analysis, which we unify as a matrix transferring process. This process enables us to transform a standard energy equation into a specific form.
It is crucial to note that this transformation hinges on the inherent nature of the equation, involving the exchange of temporally discrete information with spatially discrete information.
This process depends on the desired form of the corresponding energy equation.
 Following the transferring process, we derive the expected stability outcomes through examination of a termination index $\zeta$ and a contribution index $\rho$ . These indices explicitly elucidate the stability mechanisms of the RKSV method, making them invaluable for analyzing the various stabilities of fully discrete RKSV methods.

 Secondly, regarding the optimal order convergence of fully discrete SV schemes,  the key is to establish an error equation to present the relationship between  two consecutive temporal steps' errors. To achieve this, we first define a novel inner product $(\cdot,\cdot^*)$  and a novel bilinear form 
$a_h(\cdot,\cdot^*)$, based on the aforementioned from-trial-to-test-space mapping. With these definitions, the fully discrete SV scheme can be formulated in a Petrov-Galerkin scheme.
Using stability-like analysis, our goal is to estimate that the error $e=u-u_h$ can be decomposed into two components: $e=\xi-\eta$, where $\eta = \mathcal{P}_h u-u,\quad \xi= \mathcal{P}_h u-u_h,$ and $\mathcal{P}_h $ denotes the Lagrange interpolating operator. Given that the bound of the interpolating error $\eta$ is known, our focus shifts to establishing a relationship for  $\xi$ between two temporal steps, specifically between  $\xi^{n+1}$ and $\xi^{n}$.
 Once the error equation is established, the estimation of the  {\it energy norm} error $\3bar\xi^{n+1}\3bar^2-\3bar\xi^n\3bar^2$ reduces to bounding three types of errors: the interpolating error, the truncation error of the Taylor expansion, and the temporal difference  $\3bar\xi^{n+1}-\xi^n\3bar$. 
 
 For {\it energy norm} error, the error between two time steps can be attributed to the evolution of errors between intermediate steps. To handle this portion of error, we make full use of the properties of Runge-Kutta methods under the Strong Stability Preserving (SSP) framework to transfer energy. Consequently, the estimation of the optimal convergence order can be transformed into a matrix transferring process. Unlike stability analysis, this evolution process includes a truncation error, which necessitates the appropriate treatment of this error in each step of the matrix transferring process. Finally, we analyze the optimal convergence order of the fully discretized RKSV method under certain CFL conditions.

The paper is organized as follows. In Section 2, we introduce any-order Runge-Kutta spectral volume schemes. Section 3 discusses the properties of inner product and bilinear form. Sections 4 and 5 analyze the stability and optimal convergence orders of RKSV(s,k), respectively In Section 6, numerical examples were given to justify our theoretical findings. Finally, Section 7 presents concluding remarks.


\section{Any order Runge-Kutta spectral volume schemes}
In this section, we present any order Runge-Kutta spectral volume schemes (RKSV(s,k)) for the one-dimensional hyperbolic problem \eqref{linearEQ}-\eqref{linearEQ_boundary_x}. 

\subsection{SV method of $k$-th order in space}
We begin with a description of the spatial discretization.
First, we  partition the computational domain $I=[a,b]$
with $N+1$ points $a=x_{\frac12}<x_{\frac32}\ldots<x_{N-\frac12}<x_{N+\frac12}=b$. Each element $I_{i}=[x_{i-\frac{1}{2}},x_{i+\frac{1}{2}}], i=1,\ldots,N$
is called  as a  {\it spectral volume} and we  denote its length as $h_i=x_{i+\frac{1}{2}}-x_{i-\frac{1}{2}}$.  The partition is assumed to be {\it regular} if there exists a positive constant $\alpha$ such that $\alpha h_{max}\leq  h_{min}$, where  $h_{max}=\max (h_i,1\le i\le N)$, and $h_{min}=\min (h_i, 1\le i\le N)$. Moreover, we subdivide each spectral
volume $I_i, i\in\{1,\ldots,N\}$ with $k+2$ points $x_{i-\frac12}=x_{i,0}<x_{i,1}<\cdots<x_{i,k}<x_{i,k+1}=x_{i+\frac12}$ into
$k+1$ sub-elements $C_{i,j}=[x_{i,j},x_{i,j+1}],  j=0,\ldots, k$ which are called as {\it control volumes}.

The {\it semi-discrete} SV scheme for (\ref{linearEQ})
is a scheme to find $u_h\in \mathcal{U}_h$ such that
\begin{equation}\label{semi-scheme}
    \int_{x_{i,j}}^{x_{i,j+1}}\partial_t u_h(x,t) \mathrm{d}x+u_{i,j+1}^--u_{i,j}^-=0, \quad
    1 \leq i \leq N, 0 \leq j \leq k,
\end{equation}
where
$$u_{i,j}^-=u_h^-(x_{i,j},t),1 \leq i \leq N, 0 \leq j \leq k,$$
and the (discontinuous) finite element space $$\mathcal{U}_{h}=\left\{w\in L_2(I):\left.w\right|_{I_{i}} \in \mathbf{P}^{k}\left(I_i\right), 1 \leq i \leq N\right\},$$
with $\mathbf{P}^{k}$  the space of polynomials  of degree at most $k$.

Note that different choice of subdividing points $x_{i,j}: j=1,\ldots, k$ leads to different SV scheme. Usually, the subdividing points in each element can be derived from those in the reference interval $[-1,1]$ by an affine transformation.  Suppose $-1=y_0<y_1<y_2<\cdots<y_{k+1}=1$, we let
\[
x_{i,j}=\frac{1}{2}h_iy_j+x_i, i\in \{1,\ldots N\}, j\in \{0,\ldots,k+1\},
\]
where  $x_i=\frac{1}{2}(x_{i+\frac{1}{2}}+x_{i-\frac{1}{2}}).$

In this paper, we will consider two choices of $y_i, i=1,\ldots,k$: one use the Gauss-Legendre points, the other use the so-called {\it right-Radau} points. The SV schemes corresponding to these two choices are called as LSV
and RRSV, respectively.

\subsection{Runge-Kutta method of $s$-order in time }
Next we explain the temporal discretization for \eqref{linearEQ}. For a given positive integer $M$, we let the temporal size $\tau =T/M$ and let $t_n=n\tau,  n=0,\ldots,M$. 
To temporally discretize \eqref{semi-scheme},  we  use the following any order Runge-Kutta method.

Using any order Runge-Kutta method for solving the ODE system, as constructed in \cite{gottlieb2001strong}.
\begin{equation}\label{ODE}
    \phi_t=F(\phi).
\end{equation}
By virtue of the Shu-Osher representation \cite{gottlieb2001strong}, the general construction of the RK(s) method is given as follows. For $\ell=0,1,\ldots,s-1$ , the stage solution, advancing from $t^n$ to $t^{n+1}$, are successively sought by the variational formula
\begin{align}\label{any_order_RK}
    \phi^{n,\ell+1}=\sum_{0\le\kappa\le \ell}\left[c_{\ell \kappa}\phi^{n,\kappa}+d_{\ell \kappa}F(\phi^{n,\kappa})\right],
\end{align}
where $d_{\ell \ell}\ne 0$, $\phi^{n,0}=\phi^n$ and  $\phi^{n,s}=\phi^{n+1}$ . According to \eqref{any_order_RK}, it is evident that different explicit Runge-Kutta methods are uniquely determined by the coefficients $c_{\ell \kappa}, d_{\ell \kappa}, \kappa \in \{0, \cdots, \ell\}, \ell \in \{0, 1, \cdots, s\}$.
Moreover, for linear constant coefficient problems, all Runge-Kutta methods of the same order and the same number of intermediate steps are equivalent \cite{gottlieb2001strong}.

Under the SSP framework the coefficients in \eqref{any_order_RK} of this method can be written in two matrices,
\begin{align}\label{matrix_ssp_runge_kutta}
\left\{c_{\ell \kappa}\right\}=\left[\begin{array}{ccc;{2pt/3pt}c}
    1   &  & & \\
    & \ddots & & \\
    & & 1 & \\
    \hdashline[2pt/3pt]
    g_{s-1,0} & \cdots & g_{s-1, s-2} & g_{s-1, s-1}\\
\end{array}\right], 
\quad\left\{d_{\ell \kappa}\right\}=\left[\begin{array}{ccc;{2pt/3pt}c}
    1 & & &\\
    & \ddots & &  \\
    & &1& \\
    \hdashline[2pt/3pt]
    & & &g_{s-1, s-1}\\
\end{array}\right],
\end{align}
The matrices $\{c_{\ell \kappa}\}$ and $\{d_{\ell \kappa}\}$ are square matrices of size $(s-1)\times (s-1)$.

The parameters are defined as follows. Let $g_{00}=1$ and recursively define for $s\le 2$ that
$$ g_{s-1,\ell}=\frac{1}{\ell}g_{s-2,\ell-1},\quad \ell=1,2,\ldots,s-2,$$
with $g_{s-1,s-1}=\frac{1}{s!}$ and $g_{s-1,0}=1-\sum_{\ell=1}^{s-1}g_{s-1,\ell}$.

To apply the above Runge-Kutta method  on  \eqref{semi-scheme}, we
introduce the notation
$$\phi_{i,j}(x,t)=\int_{x_{i,j}}^{x_{i,j+1}}u(x',t)\mathrm{d}x',\quad F(\phi)_{i,j}(x,t)=u^-(x_{i,j},t)-u^-(x_{i,j+1},t),$$
for  $x\in C_{i, j}$ and $\phi=\{\phi_{i,j},i=1,\ldots,N, j=0,\ldots,k\}, F(\phi)=\{F(\phi)_{i,j},i=1,\ldots,N, j=0,\ldots,k\},$
and  rewrite the semi-discrete system \eqref{semi-scheme}  in the form of
ODE system \eqref{ODE}. An application of the scheme \eqref{any_order_RK} yields the following  RKSV(s, k) scheme for \eqref{linearEQ} : Find $v_h,u_h\in \mathcal{U}_h$, such that
\begin{align}\label{fully-discrete_1}
\left[\begin{array}{c}
\int\nolimits_{x_{i,j}}^{x_{i,j+1}}u_h^{n,1}\mathrm{d}x\\
\vdots\\
\int\nolimits_{x_{i,j}}^{x_{i,j+1}}u_h^{n,s-1}\mathrm{d}x\\
\int\nolimits_{x_{i,j}}^{x_{i,j+1}}u_h^{n,s}\mathrm{d}x\\
\end{array}\right] 
&=
\{c_{\ell \kappa}\}
\left[\begin{array}{c}
\int\nolimits_{x_{i,j}}^{x_{i,j+1}}u_h^{n,0}\mathrm{d}x\\
\vdots\\
\int\nolimits_{x_{i,j}}^{x_{i,j+1}}u_h^{n,s-2}\mathrm{d}x\\
\int\nolimits_{x_{i,j}}^{x_{i,j+1}}u_h^{n,s-1}\mathrm{d}x\\
\end{array}\right]-\tau\{d_{\ell \kappa}\}
\left[\begin{array}{c}
u_{i,j+1}^{(n,0)-}-u_{i,j}^{(n,0)-}\\
\vdots\\
u_{i,j+1}^{(n,s-2)-}-u_{i,j}^{(n,s-2)-}\\
u_{i,j+1}^{(n,s-1)-}-u_{i,j}^{(n,s-1)-}\\
\end{array}\right],
\end{align}
hold  for all $i\in \{1,\ldots,N\}, j \in\{0,\ldots, k\}$, and $\{c_{\ell k}\}$, $\{d_{\ell k}\}$ are the matrices define in \eqref{matrix_ssp_runge_kutta}.

Thus,  RKSV(s,k) is provided by \eqref{fully-discrete_1}. Next, we will proceed to analyze this fully discrete numerical scheme.


\section{Properties of inner product and bilinear form} To facilitate the subsequent analysis of the RKSV(s,k) schemes, we will analyze the properties of the inner product and bilinear form.
For this purpose, we first rewrite \eqref{fully-discrete_1} into their equivalent form which are easy to be analyzed.
We begin by  rewriting \eqref{fully-discrete_1} as a
Petrov-Galerkin method.  Introducing the {\it test} space
$$\mathcal{V}_{h}=\left\{w^{*}:\left.w^{*}\right|_{C_{i, j}} \in \mathcal{P}_{0}\left(C_{i, j}\right), 1 \leq i \leq N, 0 \leq j \leq k\right\},$$
each  function $w^*\in \mathcal{V}_{h}$ can be represented as
\begin{align}\label{w_define_first}
    w^{*}(x,t)=\sum_{i=1}^{N} \sum_{j=0}^{k} w_{i, j}^{*}(t) \chi_{\mathbf{C}_{i, j}}(x),
\end{align}
where $w_{i, j}^{*}=w_{i, j}^{*}(t)$, $1\leq i\leq N, 0\leq j\leq k$ are functions only related to the time variable $t$. $\chi_{\mathbf{C}_{i, j}} $is the characteristic function defined as $\chi_{\mathbf{C}_{i, j}}=1$ in $C_{i, j}$ and $\chi_{\mathbf{C}_{i, j}}=0$ otherwise. With these notations, the scheme \eqref{fully-discrete_1} are equivalent to find $u_h\in \mathcal{U}_h$, such that
\begin{align}\label{any_order_PG_1}
\left[\begin{array}{c}
(u_h^{n,1},w^*)\\
\vdots\\
(u_h^{n,s-1},w^*)\\
(u_h^{n,s},w^*)\\
\end{array}\right] 
&=
\{c_{\ell \kappa}\}
\left[\begin{array}{c}
(u_h^{n,0},w^*)\\
\vdots\\
(u_h^{n,s-2},w^*)\\
(u_h^{n,s-1},w^*)\\
\end{array}\right]+\tau\{d_{\ell \kappa}\}
\left[\begin{array}{c}
a_{h}(u_h^{n,0},w^*)\\
\vdots\\
a_{h}(u_h^{n,s-2},w^*)\\
a_{h}(u_h^{n,s-1},w^*)\\
\end{array}\right].
\end{align}
hold for all $w^*\in {\mathcal {V}}_h$. Here the SV bilinear form  $a_{h}(\cdot,\cdot)$ is defined as :
\begin{align}\label{bilinear_defined}
    a_{h}\left(v,w^*\right)= -\sum_{i=1}^N\sum_{j=0}^k w_{i,j}^*(v_{i,j+1}^--v_{i,j}^-),
\end{align}
where $v\in\mathcal{U}_{h},$ and $w^*\in\mathcal{V}_{h}$.

Next, we transform the previously described Petrov-Galerkin scheme into a Galerkin scheme by introducing the {\it from-trial-to-test} mapping $T$ (see \cite{zou_cao_2021}).
For all $w\in \mathcal{U}_{h}$, let $w^*={\rm T} w=\sum_{i=1}^{N} \sum_{j=0}^{k} w_{i, j}^{*}(t) \chi_{\mathbf{C}_{i, j}}(x)\in \mathcal{V}_{h}$, where $w^*_{i,j}$ are defined by
\begin{equation}\label{constant_insert}
w_{i, 0}^{*}=w_{i-\frac{1}{2}}^{+}+A_{i, 0} w_{x}\left(x_{i,0}\right), \quad w_{i, j}^{*}-w_{i, j-1}^{*}=A_{i, j} w_{x}\left(x_{i, j}\right), j \in Z_{k}
\end{equation}
Here, $A_{i,j} = \frac{h_j}{2} A_j$ for $1 \leq i \leq N$ and $0 \leq j \leq k$, where $A_j$ ($0 \leq j \leq k$) are the weights in the interpolating-type quadrature $Q(f) = \sum_{j=0}^{k+1} A_j f(y_j)$ used to calculate the integral $I(f) = \int_{-1}^1 f(y) , dy$.

 Note that since each $w \in \mathcal{U}_h$ is a polynomial of degree $k$ in each element $I_i$, and the interpolating-type quadrature $Q(f) = I(f)$ when $f$ is a polynomial of degree $k$, it follows from (\ref{constant_insert}) that
$$w_{i,k}^*=w_{i+\frac{1}{2}}^--A_{i,k+1}w_x(x_{i+\frac{1}{2}}^-).$$

With the above mapping $T$, the scheme \eqref{any_order_PG_1} are equivalent to finding $u_h\in \mathcal{U}_h$, such that
\begin{align}\label{RK_G_1}
\left[\begin{array}{c}
(u_h^{n,1},Tw)\\
\vdots\\
(u_h^{n,s-1},Tw)\\
(u_h^{n,s},Tw)\\
\end{array}\right] 
&=
\{c_{\ell \kappa}\}
\left[\begin{array}{c}
(u_h^{n,0},Tw)\\
\vdots\\
(u_h^{n,s-2},Tw)\\
(u_h^{n,s-1},Tw)\\
\end{array}\right]+\tau\{d_{\ell \kappa}\}
\left[\begin{array}{c}
a_{h}(u_h^{n,0},Tw)\\
\vdots\\
a_{h}(u_h^{n,s-2},Tw)\\
a_{h}(u_h^{n,s-1},Tw)\\
\end{array}\right].
\end{align}
hold for all $w\in {\mathcal {U}}_h$. Using the notation $w^*=Tw$, the scheme \eqref{RK_G_1} has the same representation as \eqref{any_order_PG_1}. Therefore,  we will also use the formula \eqref{any_order_PG_1}  to indicate the Galerkin scheme
\eqref{RK_G_1} throughout the rest of the paper.

To analyze \eqref{any_order_PG_1}, we will first discuss the bilinear form $a_h(\cdot,\cdot^*)$ and define the element-wise bilinear form for all $i\in \{1,\ldots,N\}$
\begin{align}\label{ebilinear_1}
a_{h,i}(v,w^*)= -\sum_{j=0}^k w_{i,j}^*(v_{i,j+1}^--v_{i,j}^-),\quad \forall v, w\in\mathcal{U}_{h}.
\end{align}

Substituting the definition of \eqref{constant_insert} into \eqref{ebilinear_1}, we obtain that (see also \cite{zou_cao_2021}):
\begin{align}\label{difference}
a_{h,i}(v,w^*)=(v,w_x)_i-v^{-}w^-|_{x_{i+\frac{1}{2}}}
+v^{-}w^+|_{x_{i-\frac{1}{2}}}-R_i(vw_x)-A_{i,0}w_x^+\left[v\right]|_{x_{i-\frac{1}{2}}}, \forall v, w\in\mathcal{U}_{h},
\end{align}
where
\begin{align}\label{Ri_define}
R_i(f) =\int_{I_i} f\mathrm{d}x-\sum_{j=0}^{k+1} A_{i,j}f(x_{i,j})
\end{align}
is the residual of the quadrature on the interval $I_i$.
Remark that for both the LSV and RRSV, we have $A_{i,0}=0$ and
$R_i(vw_x)=0$ (since on $I_i$, $vw_x$ is a polynomial of degree no more than $2k-1$), then the formula \eqref{difference} reduces to
\begin{align}\label{difference0}
a_{h,i}(v,w^*)=(v,w_x)_i-v^{-}w^-|_{x_{i+\frac{1}{2}}}
+v^{-}w^+|_{x_{i-\frac{1}{2}}}, \forall v,w\in {\mathcal {U}}_h.
\end{align}
Moreover, setting the inner product $(\cdot,\cdot^*)_i=(\cdot,\cdot^*)_{I_i}$, we have that
\begin{align}\label{inner_compact}
(v,w^*)_i=( v,w)_i+R_i(w_x\partial_x^{-1}v), \forall v,w\in \mathcal {U}_h,
\end{align}
where
\[
\partial_x^{-1}v=\int_a^x v(x',\cdot) \mathrm{d}x'.
\]
Since $w_x\partial_x^{-1}v$ is a polynomial of degree $2k$ in $I_i$,  $R_i(w_x\partial_x^{-1}v)=0$ for the RRSV scheme and $R_i(w_x\partial_x^{-1}v)$ is a constant in the $I_i$ for the LSV scheme.

\begin{remark}\label{remark_bilinear_1}
According to previous study \cite{wei_zou_2023}, both the LSV and RRSV schemes satisfy the following relationship: 
\begin{align}\label{excharge_express}
    a_h(v,w^*)+a_h(w,v^*)=-\sum_{i=1}^{N}[v]|_{x_{i+\frac{1}{2}}}[w]|_{x_{i+\frac{1}{2}}}, \forall v,w \in {\mathcal U}_h,
\end{align}
where the jump $ [w]|_{x_{i+\frac{1}{2}}}=w^+_{i+\frac{1}{2}}-w^-_{i+\frac{1}{2}}$. 

Taking $w=v$,  if $v\in {\mathcal U}_h$ satisfies $v(a)=v(b)$ or $v(a)=0$, then we have: 
\begin{align}\label{a_DG_estimate0}
    a_h(v,v^*)\le  0 .
\end{align}

Moreover, the following estimate holds:
\begin{align}\label{a_DG_estimate}
    |a_h(v,w^*)|\lesssim  h^{-1}\|v\|\|w\|,\quad \forall v,w \in {\mathcal U}_h,
\end{align}
where $a_h(v,w^*)=\sum_{i=1}^{N}a_{h,i}(v,w^*)$ .
\end{remark}

\begin{remark}\label{remark_positive_matrix}
According to the study \cite{xu_zhang_2019}, 
if $G=\{g_{pq}\}$ be a symmetric positive semidefinite matrix for the row numbers and column numbers are both taken from a given set $\mathcal{G}$. Then the following inequality holds:
\begin{align}\label{positive_matrix_1}
    \sum_{p\in \mathcal{G}}\sum_{q\in \mathcal{G}}g_{pq}a_h(v_p,v_q^*)\le 0
\end{align} 
\end{remark}

\begin{remark}\label{remark_inner_productive}
According to previous study \cite{zou_cao_2021},  the new norm can be defined based on the inner product (Eq. \ref{inner_compact}) as follows:
\begin{align}\label{norm_new}
    \3barw\3bar^2=(w,w^*)=(w,w)+R(w_x\partial_x^{-1}w).
\end{align}
Here, $\3bar \cdot \3bar$ represents the new norm of $w$.

Furthermore, there exists an estimate relating the $L_2$ norm and the new norm:
\begin{align}\label{norm_l2_new}
    \|w^*\|\lesssim \|w\| \lesssim \3bar w\3bar,
\end{align}
This implies that the $L_2$ norm and the new norm are equivalent.

Moreover, we can prove the symmetry property of the inner product:
\begin{align}\label{inner_excharge_2}
    (w_1,w_2^*)=(w_2,w_1^*),\  \forall w_1,w_2\in {\mathcal U}_h.
\end{align}
This property ensures that the inner product is symmetric.
\end{remark}

\section{Stability}\label{stability_any_order_section}
In this section, we will analyze the stability of the  RKSV(s,k) scheme \eqref{any_order_PG_1}.

For the stage solutions $u_h^{n,\ell}=u_h(t^{n,\ell})$ with $\ell =0,1,\cdots,s$ after the time level $t^n$, following the works of \cite{zhang2004error,zhang_shu_2010_error_3order}, we define a series of the temporal differences in the form :
\begin{align}\label{time_new_1}
\mathbb{D_\kappa}u^n_h=\sum_{0\le l \le \kappa} \sigma_{\kappa l}u^{n,l}_h,\quad 1\le \kappa \le s,
\end{align}
where $\sum_{0\le l \le \kappa}\sigma_{\kappa l}=0 $ and the initial condition is given by $\mathbb{D}_0u^n_h=u^n_h$.   $\mathbb{D}\kappa$ acts on $u_h^n$, it can be seen as a linear representation of the stage solutions $u_h^{n,\ell}$.

It is worth noting that up to this point, $\sigma_{\kappa \ell}$ remains unknown. In order to further determine its value, we will leverage the relationship between inner products and bilinearity. Let the operator $\mathbb{D}_\kappa$ satisfy the following expression:
\begin{align}\label{time_new_2}
(\mathbb{D}_\kappa u^n_h, w^*)=\tau a_h(\mathbb{D}_{\kappa-1}u_h^n,w^*).
\end{align}
Based on formula \eqref{time_new_2}, we can recursively obtain the values of $\sigma_{\kappa \ell}$, which is the expression for $\mathbb{D}_\kappa$ from equation \eqref{time_new_1}.

Furthermore, in Eq\eqref{time_new_2}, the left inner product contains information regarding time discretization, while the right-hand side contains information regarding spatial discretization, expressed through the bilinear form $a_h(\cdot,\cdot^*)$.


\begin{lemma}\label{operator_k_k-1}
For $\kappa=1,2\cdots,s$, there holds
\begin{align}\label{time_new_3}
    \3bar \mathbb{D_\kappa}u^n_h\3bar \le C\lambda \3bar \mathbb{D}_{\kappa-1}u^n_h\3bar,\quad \forall n\ge 0,
\end{align} 
where $\lambda=\frac{\tau}{h}$ is the CFL condition number.
\end{lemma}	
\begin{proof}
Taking $w^*=\mathbb{D}_\kappa u_h^{n*}$ in the \eqref{time_new_2}, we have
\begin{align}\label{lemma_1}
    (\mathbb{D}_\kappa u^n_h, \mathbb{D}_\kappa u_h^{n*})=\tau a_h(\mathbb{D}_{\kappa-1}u_h^n,\mathbb{D}_\kappa u_h^{n*}).
\end{align}
Using the definition \eqref{norm_new} and the estimate of $a_h(\cdot,\cdot^*)$ in \eqref{a_DG_estimate}, we obtain
\begin{align*}
    \3bar\mathbb{D}_\kappa u^n_h\3bar^2\le C \frac{\tau}{h} \|\mathbb{D}_{\kappa-1}u_h^n\|\|\mathbb{D}_\kappa u_h^{n*}\|.
\end{align*}
Using the equivalence \eqref{norm_l2_new} between $L_2$ norm $\|\cdot\|$ and the norm $\3bar\cdot \3bar$, we have 
\begin{align}
    \3bar\mathbb{D}_\kappa u^n_h\3bar^2\le C \lambda \3bar\mathbb{D}_{\kappa-1}u_h^n\3bar\3bar\mathbb{D}_\kappa u_h^{n}\3bar.
\end{align}
If we remove the identical terms from the above equation, this proves the Lemma.
\end{proof}

Below we will use the generalized notations
\begin{align}
u^{n,\kappa+ms}=u^{n+m,\kappa},\quad \kappa=0,1,\ldots,s-1,
\end{align}
for any given integer $m\ge 1$. Here $n$ and $n+m$ are called the time levels, $\kappa$ and
$\kappa+ms$ are called the stage numbers, and the $m$ called the step number. When $m>1$, the Runge-Kutta method is a multistep method. In the subsequent analysis of Runge-Kutta methods, we will consider $m=1$, implying a single-step Runge-Kutta method.

In the above process to define the temporal differences, we also achieve the evolution identity
\begin{align}\label{time_trans_1}
\alpha_0u^{n+1}_h=\sum_{0\le p\le s}\alpha _p\mathbb{D}_p u^n_h,
\end{align} 
where $\alpha> 0$ is used only for scaling. For convenience, denote $\boldsymbol{\alpha}=(\alpha_0,\alpha_1,\ldots,\alpha_{s}) $.

For the $u^{n+1}_h,\mathbb{D}_\kappa u_h^{n}\in \mathcal{U}_{h}$ $\kappa\in\{0,1,\cdots,s-1\}$, according to \eqref{w_define_first}, we can obtain $(u^{n+1}_h)^*,(\mathbb{D}_\kappa u_h^{n})^*\in \mathcal{V}_{h}$  $\kappa\in\{0,1,\cdots,s-1\}$. Therefore, by taking the inner product of the left side of \eqref{time_trans_1} with $\alpha_0 (u^{n+1}_h)^* $ and the right side of \eqref{time_trans_1} with $\sum_{0\le q\le s}\alpha_q (\mathbb{D}_q u^n_h)^*$, we have:
\begin{align*}
\alpha_0^2(u^{n+1}_h,(u^{n+1}_h)^*)=(\sum_{0\le p\le s}\alpha _p\mathbb{D}_p u^n_h, \sum_{0\le q\le s}\alpha _q(\mathbb{D}_q u^{n}_h)^*).
\end{align*} 
Using the symmetry property of the inner product\eqref{inner_excharge_2}, we have the energy equation:
\begin{align}\label{stability_1}
\alpha^2_0(\3baru^{n+1}_h\3bar^2-\3baru^n_h\3bar^2)=\sum_{0\le p,q\le s}a_{pq}(\mathbb{D}_pu^n_h,(\mathbb{D}_qu^n_h)^*)=\text{RHS}(u_h^n),
\end{align}	
where $a_{00}=0$, and $a_{pq}=\alpha_p\alpha_q$ if $p+q>0$.

According to \eqref{stability_1}, the stability of RKSV(s,k) is determined by RHS$(u_h^n)$. Next, let's discuss the magnitude of RHS$(u_h^n)$.
Since the right-hand side term of equation \eqref{stability_1} depends solely on inner products of time stage solutions, we attempt to utilize the equivalence relation \eqref{time_new_2} to transform some of the inner products of time stage solutions into bilinear forms discretized in space. After applying the equivalence relation $l$ times, where $l\geq0$, we obtain RHS$^{(l)}(u_h^n)$ as:
\begin{align}\label{stability_l}
\text{RHS}^{(l)}(u_h^n)=\sum_{0\le p,q\le s}a_{pq}^{(l)}(\mathbb{D}_pu^n_h,(\mathbb{D}_qu^n_h)^*)+\sum_{0\le p,q\le s}b_{pq}^{(l)}\tau a_h(\mathbb{D}_pu^n_h,(\mathbb{D}_qu^n_h)^*).
\end{align}	
Here, we refer to $l$ as the number of transformations applied to RHS$(u_h^n)$. $a^{(l)}_{pq}$ and $b^{(l)}_{pq}$ denote the elements in the $p$-th row and $q$-th column of the matrices $\mathbb{A}^{(l)}$ and $\mathbb{B}^{(l)}$, respectively, where $p, q \in \{0,1,\cdots,s\}$.
Based on this, we will provide a detailed explanation of the evolution process of RHS$^{(l)}(u_h^n)$.

\subsection{The evolution iteration of RHS$(u_h^n)$}
According to the research by Xu and Zhang \cite{xu_zhang_2019} on any-order RKDG schemes, the evolution iteration of RHS$(u_h^n)$ in the RKSV(s,k) schemes follows a similar process.

Firstly, we will provide the initial matrix for RHS$(u_h^n)$, the process of which is described by equation \eqref{stability_1}, as follows:
$$\mathbb{A}^{(0)}=\{a_{pq}\},\quad \mathbb{B}^{(0)}=\mathbb{O}. $$
Here, $\mathbb{A}^{(0)}$ represents an initial matrix with elements $a_{pq}$, and $\mathbb{B}^{(0)}$ represents an initial matrix which is the zero matrix denoted as $\mathbb{O}$. Moreover, it is mentioned in the study that $\mathbb{A}^{(0)}$ and $\mathbb{B}^{(0)}$ are symmetric matrices.

The motivation behind the matrix transformation in equation \eqref{stability_l} can be attributed to two reasons. Firstly, the equivalence relation \eqref{time_new_2} transforms some of the inner products of time stage solutions into bilinear forms discretized in space, which means certain elements of matrix $\mathbb{A}$ undergo specific transformations to reach matrix $\mathbb{B}$. Secondly, it is driven by fully exploiting the approximate skew-symmetric property of bilinear forms discretized in space, as described in equation \eqref{excharge_express}.


Below, we summarize the evolution of matrices, assuming that starting from the $l$-th iteration, where $l\geq 1$, the matrices $\mathbb{A}^{(l)}$ and $\mathbb{B}^{(l)}$ in the expression for RHS$^{(l)}(u_h^n)$ are:
\begin{align*}
&\begin{aligned}
    &\mathbb{A}^{(l)}= \left\{a_{pq}^{(l)}\right\}=\left[\begin{array}{c:  c:  ccc}
        \mathbb{O} & \mathbb{O} & \mathbb{O} & \cdots & \mathbb{O} \\
        \hdashline \mathbb{O} & a_{l, l}^{(l)} & a_{l, l+1}^{(l)} & \cdots & a_{l, s}^{(l)} \\
        \hdashline \mathbb{O} & a_{l+1, l}^{(l)} & a_{l+1, l+1}^{(l)} & \cdots & a_{l+1,s}^{(l)} \\
        \vdots & \vdots & \vdots & \ddots & \vdots \\
        \mathbb{O} & a_{s, l}^{(l)} & a_{s, l+1}^{(l)} & \cdots & a_{s, s}^{(l)}
    \end{array}\right],
    \end{aligned}\\
    &\begin{aligned}
    &\mathbb{B}^{(l)}= \left\{b_{pq}^{(l)}\right\}=\left[\begin{array}{c:  c:  ccc}
        \star & \star & \star & \cdots & \star \\
        \hdashline \star & b_{l, l}^{(l)} & b_{l, l+1}^{(l)} & \cdots & b_{l, s}^{(l)} \\
        \hdashline \star & b_{l+1, l}^{(l)} & 0 & \cdots & 0 \\
        \vdots & \vdots & \vdots & \ddots & \vdots \\
        \star & b_{s, l}^{(l)} & 0 & \cdots & 0
    \end{array}\right].
    &
\end{aligned}
\end{align*}

It is worth noting that although we represent the zero matrix as $\mathbb{O}$ in the matrix $\mathbb{A}^{(l)}$, different positions of the zero matrix $\mathbb{O}$ represent different meanings. The zero matrix on the diagonal is denoted as $\mathbb{O}_{(l+1) \times (l+1)}$, the zero matrix from the second to the $s$-th element in the first row is denoted as $\mathbb{O}_{(l+1) \times 1}$, and the zero matrix from the second to the $s$-th element in the first column is denoted as $\mathbb{O}_{1 \times (l+1)}$. The same principle applies to the symbol $\star$; we use $\star$ to represent non-zero elements in the matrix $\mathbb{B}^{(l)}$.

Next, we will evolve RHS$^{(l)}(u_h^n)$ one step further to RHS$^{(l+1)}(u_h^n)$ using the equivalence relation \eqref{time_new_2}. Specifically, we will transform the elements in the $(l+1)$-th row of the matrix $\mathbb{A}^{(l)}$ into the $(l+1)$-th row of the matrix $\mathbb{B}^{(l)}$. According to the symmetry of matrices, similar treatment will be applied to the $(l+1)$-th column of $\mathbb{A}^{(l)}$.

First, we check if $a_{l,l}^{(l)}$ equals zero. If it is zero, we can utilize the symmetry of the inner product $(\cdot, \cdot^*)$, leading to:
\begin{align*}
    a_{l+1,l}^{(l)}(\mathbb{D}_{l+1} u^n_h, (\mathbb{D}_{l} u^n_h)^*)+a_{l,l+1}^{(l)}(\mathbb{D}_{l} u^n_h, (\mathbb{D}_{l+1} u^n_h)^*)=2a_{l+1,l}^{(l)}(\mathbb{D}_{l+1} u^n_h, (\mathbb{D}_{l} u^n_h)^*).
\end{align*}
Then by making full use of relationship \eqref{time_new_2} among those temporal differences, we obtain
\begin{align}\label{diagonal_value}
    2a_{l+1,l}^{(l)}(\mathbb{D}_{l+1} u^n_h, (\mathbb{D}_{l} u^n_h)^*)=	2\tau a_{l+1,l}^{(l)} a_h(\mathbb{D}_{l} u^n_h,(\mathbb{D}_{l} u^n_h)^*),
\end{align}
and for $l+1\le p\le s-1$
\begin{align}
    &a_{p+1,l}^{(l)}( \mathbb{D}_{p+1} u^n_h, (\mathbb{D}_{l} u^n_h)^*)+a_{l,p+1}^{(l)}( \mathbb{D}_{l} u^n_h, (\mathbb{D}_{p+1} u^n_h)^*)
    +\delta_{p,l+1} a_{p,l+1}^{(l)}( \mathbb{D}_{p} u^n_h, (\mathbb{D}_{l+1} u^n_h)^*)\label{l_terms_involve}\\
    &= \delta_{p,l+1} \left[a_{p,l+1}^{(l)}-\frac{2a_{p+1,l}^{(l)}}{\delta_{p,l+1}}\right](\mathbb{D}_pu_h^n,(\mathbb{D}_{l+1}u_h^n)^*) +2a_{p+1,l}^{(l)}\tau a_h(\mathbb{D}_pu_h^n,(\mathbb{D}_{l}u_h^n)^*)\nonumber\\
&\quad\quad\quad +2a_{p+1,l}^{(l)}\tau a_h(\mathbb{D}_l u_h^n,(\mathbb{D}_{p}u_h^n)^*).\nonumber
\end{align}
Here, when $p=l+1$, $\delta_{p,l+1}$ equals 1, and in other cases, it equals 2. Thus, we have completed the processing of the $l$-th row, indicating the end of the $(l+1)$-th iteration.

Next, we will provide a systematic explanation of the evolution process of the matrix.
\begin{theorem}\label{matrix_theorem_main}
For $l \geq 1$, the iterative evolution of the matrices \(\mathbb{A}\) and \(\mathbb{B}\) is defined by the following steps:
\begin{itemize}
\item Firstly, initialize the matrices as $\mathbb{A}^{(0)}=\{a_{pq}\}$, $\mathbb{B}^{(0)}=\mathbb{O}$.
\item Secondly, iterate for $\ell \geq 1$ to update the matrices.
Update $\mathbb{A}^{(l)}=\{a_{pq}^{(l)}\}$ using the recursive formula:
 \begin{align}\label{A_transformation}
        a_{pq}^{(l)}=a_{qp}^{(l)}= \begin{cases}0, & 0 \leq q \leq l-1, \\
            a_{pq}^{(l-1)}-2 a_{p+1, q-1}^{(l-1)}, & p=l \text { and }  q=l, \\ 
            a_{p q}^{(l-1)}-a_{p+1, q-1}^{(l-1)}, & l+1 \leq p\leq  s-1 \text { and } q=l, \\ a_{pq}^{(l-1)}, & \text {otherwise }.\end{cases}
    \end{align}
Update $\mathbb{B}^{(l)}=\{b_{pq}^{(l)}\}$ using the recursive formula:
 \begin{align}\label{B_transformation}
        b_{pq}^{(l)}=b_{qp}^{(l)}= \begin{cases}2 a_{p+1, q}^{(l-1)}, & l-1 \leq p \leq  s-1 \text { and } q=l-1 \\ 
        b_{pq}^{(l-1)}, & \text {otherwise. }\end{cases}
    \end{align}

\item Third, update $l-1$ to $l$. Then, check if $a_{ll}^{(l)}$ equals 0. If it does, repeat the second step. If it does not equal 0, the matrix evolution terminates, and the termination index $\zeta$ is set to $l$.
\item Lastly, the algorithm outputs the final matrices $\mathbb{A}^{(\zeta)}$ and $\mathbb{B}^{(\zeta)}$, as well as the termination index $\zeta$. These values represent the resulting transformed matrices and the index at which the transformation process is completed.
\end{itemize}
\end{theorem}
\begin{remark}
According to Theorem \ref{matrix_theorem_main}, describing the evolution process of the RHS$(u_h^n)$ in the fully discrete RKSV(s,k) numerical scheme, we are surprised to find that the coefficients in the matrices $\mathbb{A}^{(l)}$ and $\mathbb{B}^{(l)}$ obtained from the RKSV(s,k) method are consistent with those obtained in the RKDG(s,k) method described in literature \cite{xu_zhang_2019}. This consistency can be attributed to the following two key factors:
\begin{enumerate}
    \item The inner product $(\cdot, \cdot^*)$ defined in RKSV(s,k) possesses symmetry, and the norm derived from this inner product is denoted as $\3bar \cdot \3bar$. In other words, for this inner product, differences between different formats can be represented by the inner product itself, rather than reflected in the coefficients before the inner product. Thus, the expression for the matrix $\mathbb{A}$ can maintain consistency.
\item In RKSV(s,k), we mainly consider two specific types of SV methods, namely LSV and RRSV. The bilinear form in these two forms includes an additional constant residual term compared to the bilinear form in RKDG. This implies that the bilinear form used in RKSV is essentially consistent with that used in RKDG, with the addition of a constant residual term.
\end{enumerate}
These two factors ensure the consistency in the matrix evolution between RKSV and RKDG formats.
\end{remark}

\subsection{Stability of the RKSV(s,k)}
To analyze the stability of RKSV(s,k), we need to analyze the term RHS$^{(\zeta)}(u_h^n)$, which involves both temporal and spatial discrete terms. This analysis follows the study conducted in \cite{xu_zhang_2019}.
\begin{align}\label{analysis_rhs}
    \alpha^2_0(\3baru^{n+1}_h\3bar^2-\3baru^n_h\3bar^2)&=\sum_{0\le p,q\le s}a_{pq}^{(\zeta)}(\mathbb{D}_pu^n_h,(\mathbb{D}_qu^n_h)^*)+\sum_{0\le p,q\le s}b_{pq}^{(\zeta)}\tau a_h(\mathbb{D}_pu^n_h,(\mathbb{D}_qu^n_h)^*)\\
    &=\text{RHS}^{(\zeta)}(u_h^n)\nonumber.
\end{align}

For the analysis of RHS$^{(\zeta)}(u_h^n)$, we will divide it into two parts: one part representing the temporal discretization term represented by the matrix $\mathbb{A}$, and the other part representing the spatial discretization term represented by the matrix $\mathbb{B}$.

For the temporal discretization term, utilizing the newly defined norm \eqref{norm_new}, due to $a^{(\zeta)} _{\zeta\zeta} \neq 0$, we have:
\begin{align}\label{stability_first_trem}
    \sum_{0\le p,q\le s}a_{pq}^{(\zeta)}(\mathbb{D}_pu^n_h,(\mathbb{D}_qu^n_h)^*)\le \left[a^{(\zeta)} _{\zeta\zeta}+\lambda \mathcal{Q}(\lambda)\right]\3bar \mathbb{D}_\zeta u^n_h\3bar^2,
\end{align}
where $\lambda=\frac{\tau}{h}$ is the CFL constant, and $\mathcal{Q}(\lambda)$ is a polynomial in $\lambda$ with non-negative coefficients.

Next, let's expand the estimation for the spatial discretization term:
$$\sum_{0\leq p,q\leq s} b_{pq}^{(\zeta)}\tau a_h(\mathbb{D}_p u^n, (\mathbb{D}_q u^n)^*).$$ 
The estimation for this part can be conducted following the method proposed in \cite{xu_zhang_2019}, with a focus on the $(\kappa+1)$-th order leading principal submatrix $\mathbb{B}^{\zeta}_\kappa$ of $\mathbb{B}^{(\zeta)}$. Let det$\mathbb{B}^{\zeta}_\kappa$ denote the value of the principal subdeterminant. Define the set $\mathcal{B}=\{\kappa: \text{det} \mathbb{B}^{\zeta}_\kappa \leq 0, \text{ and } 0 \leq \kappa \leq \zeta-1\}$. Then, the indicator factor $\rho$ is:
\begin{align}\label{rho_define_1_main}
    \rho= 
\begin{cases} min\{\gamma:\gamma\in \mathcal{B}\} & \text{if} \quad \mathcal{B} \ne 0,\\ 
    \zeta & \text { otherwise. }
\end{cases}
\end{align}
Then, we divide the set $\pi=\{0,1,\ldots,s\}$ into three parts: $\pi_1=\{0,1,\ldots,\rho-1\}$, $\pi_2=\{\rho,\rho+1,\ldots,\zeta-1\}$, and $\pi_3=\{\zeta,\zeta+1,\ldots,s\}$. It is worth noting that when $\rho$ is set to 0, the set $\pi$ is divided into two parts $\pi_2$ and $\pi_3$. Similarly, when $\rho$ is set to $\zeta$, the set $\pi$ is divided into two parts $\pi_1$ and $\pi_3$.

Next, we will estimate the spatial discretization terms in different sets:
\begin{align}\label{invided_domain_1}
    \sum_{0\le p,q\le s}b_{pq}^{(\zeta)}\tau a_h(\mathbb{D}_pu^n_h,(\mathbb{D}_qu^n_h)^*)=\sum_{ a,b=1,2,3}\mathcal{A}_{ab},
\end{align}
where
\begin{align}\label{define_A_ab}
    \mathcal{A}_{ab}=\sum_{i\in \pi_a,j\in \pi_b}\tau b_{ij}^{(\zeta)} a_h(\mathbb{D}_iu^n_h,(\mathbb{D}_ju^n_h)^*).
\end{align}

According to the definition of $\rho$ \eqref{rho_define_1_main}, the leading principal submatrix $\mathcal{B}^{\zeta}_{\rho-1}$ is positive definite. Therefore, there exists a constant $\epsilon > 0$ such that $\mathcal{B}^{\zeta}_{\rho-1}-\epsilon E_{\rho-1}$ is positive semi-definite, where $E$ is the identity matrix. Then, by Lemma \ref{remark_positive_matrix}, we have:
\begin{align*}
    \mathcal{A}_{11}=\sum_{i\in \pi_1,j\in \pi_1}\tau( b_{ij}^{(\zeta)}-e_{i,j}) a_h(\mathbb{D}_iu^n_h,(\mathbb{D}_ju^n_h)^*)\lesssim0
\end{align*}
here $e_{i,j}$ is the element in row $i$ and column $j$ of matrix $\epsilon E$.
Then, by the properties of the bilinear form \eqref{a_DG_estimate0}, we can get :
\begin{align}\label{invided_domain_11}
     \mathcal{A}_{11}\lesssim-\epsilon \tau h^{-1} \sum _{i\in \pi_1}\3bar\mathbb{D}_iu^n_h\3bar_{\Gamma_h}^2 .
\end{align}
Then, according to Lemma \ref{remark_bilinear_1}, \ref{operator_k_k-1}, the relationship of norms \eqref{norm_l2_new}, 
\begin{align*}
    \mathcal{A}_{12}+\mathcal{A}_{21}&=\sum_{i\in \pi_1,j\in \pi_2}\tau b_{ij}^{(\zeta)} \left(a_h(\mathbb{D}_iu^n_h,(\mathbb{D}_ju^n_h)^*)+ a_h(\mathbb{D}_ju^n_h,(\mathbb{D}_iu^n_h)^*)\right)\nonumber\\
    &\lesssim \tau \sum_{i\in \pi_1}\sum_{j\in \pi_2} b_{ij}^{(\zeta)}\|\mathbb{D}_iu^n_h\|_{\Gamma_h} \|\mathbb{D}_ju^n_h\|_{\Gamma_h}.
    \end{align*}
 
By the Young's inequality  and the relationship among temporal differences \eqref{operator_k_k-1}, there exists an $\epsilon_1$ such that we have the estimation:
\begin{align}\label{invided_domain_12_21}
    \mathcal{A}_{12}+\mathcal{A}_{21}
   & \lesssim \frac{\epsilon_1}{4} \tau h^{-1} \sum_{i\in\pi_1}\3bar\mathbb{D}_iu^n_h\3bar^2+ \lambda \mathcal{Q}_1(\lambda)\3bar \mathbb{D}_\rho u_h^n \3bar^2.
\end{align}
where $\mathcal{Q}_1 (\lambda)$ is a polynomial in $\lambda$ with non-negative coefficients.

Similarly, we have:
\begin{align}\label{invided_domain_22_23}
    \mathcal{A}_{22}+\mathcal{A}_{23}+\mathcal{A}_{32}\lesssim \lambda \mathcal{Q}_1(\lambda)\3bar \mathbb{D}_\rho u_h^n \3bar^2+\lambda \mathcal{Q}_2(\lambda)\3bar \mathbb{D}_\zeta u_h^n \3bar^2,
\end{align}
where $\mathcal{Q}_2 (\lambda)$ is a polynomial in $\lambda$ with non-negative coefficients. 

And  there exists $\epsilon_2$, and we have the estimation:
\begin{align}\label{invided_domain_13_31}
    \mathcal{A}_{13}+\mathcal{A}_{31}\lesssim \frac{\epsilon_2}{4} \tau h^{-1} \sum_{i\in\pi_1}\3bar\mathbb{D}_iu^n_h\3bar^2+\lambda \mathcal{Q}_2(\lambda)\3bar \mathbb{D}_\zeta u_h^n \3bar^2.
\end{align}

Therefore, combining inequalities \eqref{invided_domain_11}-\eqref{invided_domain_13_31}, we have the estimation for the spatial discretization term as:
\begin{align}\label{stability_second_trem}
    \sum_{0\le p,q\le s}b_{pq}^{(l)}\tau a_h(\mathbb{D}_pu^n_h,(\mathbb{D}_qu^n_h)^*)\lesssim \frac{\epsilon}{2} \tau h^{-1} \sum_{i\in\pi_1}\3bar\mathbb{D}_iu^n_h\3bar^2+(\lambda \mathcal{Q} +\lambda \mathcal{Q}_2(\lambda))\3bar\mathbb{D_{\zeta}}u^n_h\3bar^2+\lambda \mathcal{Q}_1(\lambda))\3bar\mathbb{D_{\rho}}u^n_h\3bar^2,
\end{align}
where $\epsilon= max\{\epsilon_1,\epsilon_2\}$.
Thus, combining \eqref{stability_first_trem} , \eqref{invided_domain_11} and \eqref{stability_second_trem}, we obtain:
\begin{align}\label{result_stability}
    \alpha_0^2(\3baru^{n+1}_h\3bar^2-\3baru^{n}_h\3bar^2)\le \left[a_{\zeta\zeta}^{(\zeta)}+\lambda \mathcal{Q}(\lambda)+\lambda \mathcal{Q}_2(\lambda)\right]\3bar\mathbb{D_{\zeta}}u^n_h\3bar^2+\lambda \mathcal{Q}_1(\lambda))\3bar\mathbb{D_{\rho}}u^n_h\3bar^2.
\end{align}

Before providing the stability analysis of  RKSV(s,k), let's give definitions of monotonic stability and weak stability.

\textbf{Monotonicity Stability:}
A scheme is said to have monotonicity stability if $\3baru^{n+1}_h\3bar $ satisfy, 
\[
\3baru^{n+1}_h\3bar^2\le \3baru^{n}_h\3bar^2, \quad n\ge 0.
\]

\textbf{Weak($\gamma$) Stability:}
A scheme is said to have weak($\gamma$) stability if $\3baru^{n+1}_h\3bar $ satisfy,
\[
\3baru^{n+1}_h\3bar^2\le (1+C\lambda^{\gamma})\3baru^{n}_h\3bar^2, \quad n\ge 0.
\]
The CFL number $\lambda$ is small enough, and the constant $C>0$ is independent of $\tau , h$ and $n$.

Based on the definitions of stability mentioned above, we present the  conclusions regarding the stability of RKSV(s,k). 
\begin{theorem} With the termination index $\zeta$ and the contribution index $\rho$ obtained by the above matrix transferring process, we have the following stability of  RKSV(s,k).
    \begin{enumerate}
        \item If $a_{\zeta \zeta}^{(\zeta)}<0$ and $\rho=\zeta$, then the scheme has monotonicity stability.
        \item If $a_{\zeta \zeta}^{(\zeta)}<0$ and $\rho<\zeta$, then the scheme has weak $(2 \rho+1)$ stability.
        \item If $a_{\zeta \zeta}^{(\zeta)}>0$, then the scheme has weak $(\gamma)$ stability with $\gamma=\min (2 \zeta, 2 \rho+1)$.
    \end{enumerate}
\end{theorem}
\begin{proof}
Since $a_{\zeta \zeta}^{(\zeta)}<0$ and $\rho=\zeta$, if the CFL number $\lambda$ is small enough, we can get
    $$
    \alpha_0^2(\3baru^{n+1}_h\3bar^2-\3baru^{n}_h\3bar^2)=\left[a_{\zeta \zeta}^{(\zeta)}+\lambda \mathcal{Q}_1(\lambda)+\lambda \mathcal{Q}_2(\lambda)\right]\3bar\mathbb{D}_\zeta u^n_h\3bar^2 \leq 0
    $$
  This implies the first conclusion.
    
    If $a_{\zeta \zeta}^{(\zeta)}<0$ and $\rho<\zeta$, we can still keep the nonpositivity as above if the CFL number is small enough. As a result, we can get from Lemma 3.4 that
    $$
    \alpha_0^2(\3baru^{n+1}_h\3bar^2-\3baru^{n}_h\3bar^2) \leq C \lambda\3bar\mathbb{D}_\rho u^n_h\3bar^2 \leq C \lambda^{2 \rho+1}\3baru^n_h\3bar^2
    $$
which implies the second conclusion.

The last conclusion can be obtained along the same lines, so the proof is omitted.
\end{proof}

\section{Analysis of the error estimate}
In this section, we will first present the error equation of RKSV(s,k), and then estimate its convergence order optimally. Before that, let's introduce the Lagrange projection operator.

\subsection{Lagrange projection }
Introducing the Lagrange projection  $\mathcal{P}_h:\mathcal{H}_h \to \mathcal{U}_h$  which satisfies the $k+1$ conditions
\begin{align}\label{L_projection_define}
    \mathcal{P}_hu(x_{i,j})=u(x_{i,j}), i\in\{1,\ldots,N\}, j\in\{1,\ldots,k+1\}.
\end{align}
By the fact, there holds the interpolating property: 
\begin{align}\label{p_estimate_2}
    \|u-{\mathcal P}_hu\|\lesssim h^{k+1}\|u\|_{k+1,\infty}.
\end{align}

When $u$ is the true solution of the hyperbolic equation and $u_h$ is the numerical solution, the error $e_u=u-u_h$ can be decomposed into:
\begin{align}\label{error_divided_2}
    e_u=\xi_u-\eta_u,
\end{align}
where
\[
\eta_u = \mathcal{P}_h u-u,\quad \xi_u= \mathcal{P}_h u-u_h.
\]
From the above definitions, we  have $\xi_u\in {\mathcal U}_h$, while
$e_u,\eta_u\in {\mathcal H}_h$.
\begin{lemma}\label{middle_1_2}
    By the definition of $a_h(\cdot,\cdot^*)$ in \eqref{bilinear_defined}, then for all $w\in {\mathcal U}_h$,
    \begin{align}\label{eta_probility}
        a_h(\eta_u, w^*)=0.	
    \end{align}
\end{lemma}
\begin{proof}
    Taking $v=\eta_u$ in \eqref{ebilinear_1}, then we have
    \[
    a_h(\eta_u,w^*)=-\sum_{i=1}^N\sum_{j=0}^k w_{i,j}^*((\eta_u)_{i,j+1}^--(\eta_u)_{i,j}^-).
    \]
    By the \eqref{L_projection_define}, we obtain in each point $x_{i,j} (i\in\{1,\ldots,N\}, j\in\{1,\ldots,k+1\} )$ satisfied $(\eta_u)_{i,j}=\eta_u(x_{i,j})=0$. This way, the lemma is proved.	
\end{proof}

\subsection{Error equation}
Let's denote the exact solution function $u$ at each time stage as $u \in \mathcal{H}_h$. 
\begin{align*}
    u^{n,\ell}=u^{(\ell)}(t^n),\quad \ell=0,1,\cdots,s,
\end{align*} 
where $u^{n,0}=u^{n}$ and $ u^{n,s}=u^{n+1}$.

There exist $u^{n,\ell}\in \mathcal{H}_h$, where $n\in \{0,1,\cdots,M\}$ and $\ell \in \{0,1\cdots,s\}$, such that for any $w\in \mathcal{U}_h$, we have:
\begin{align}\label{any_order_exact_1}
\left[\begin{array}{c}
(u^{n,1},w^*)\\
\vdots\\
(u^{n,s-1},w^*)\\
(u^{n,s},w^*)\\
\end{array}\right] 
&=
\{c_{\ell \kappa}\}
\left[\begin{array}{c}
(u^{n,0},w^*)\\
\vdots\\
(u^{n,s-2},w^*)\\
(u^{n,s-1},w^*)\\
\end{array}\right]+\tau\{d_{\ell \kappa}\}
\left[\begin{array}{c}
a_{h}(u^{n,0},w^*)\\
\vdots\\
a_{h}(u^{n,s-2},w^*)\\
a_{h}(u^{n,s-1},w^*)\\
\end{array}\right]
+\left[\begin{array}{c}
0\\
\vdots\\
0\\
(G^{n,s},w^*)\\
\end{array}\right]
\end{align}
where $G^{n,s}$ represents the truncation error in the temporal discretization. It is worth noting that in the equation \eqref{any_order_exact_1} satisfied by the exact solution, apart from the truncation error $G^{n,s}$ in the last row, which is non-zero, for RK(s) numerical schemes, temporal truncation error only manifests in the final step.

By the Taylor expansion, there exist a $\theta\in (t^n,t^{n+1})$ such that $G^{n,s}=\frac{\tau^{s+1}}{(s+1)!}\partial_t^{s+1}u(\cdot,\theta)$. Consequently, we always have
\begin{align}\label{time_truncation}
    |G^{n,s}|\le C\tau^{s+1}\|u\|_{s+1,\infty}
\end{align}

By subtracting the equation satisfied by the numerical solution \eqref{any_order_PG_1} from the equation satisfied by the exact solution \eqref{any_order_exact_1}, we obtain the equation satisfied by the error $e_u$:
\begin{align}\label{any_order_e_u_1}
\left[\begin{array}{c}
(e_u^{n,1},w^*)\\
\vdots\\
(e_u^{n,s-1},w^*)\\
(e_u^{n,s},w^*)\\
\end{array}\right] 
&=
\{c_{\ell \kappa}\}
\left[\begin{array}{c}
(e_u^{n,0},w^*)\\
\vdots\\
(e_u^{n,s-2},w^*)\\
(e_u^{n,s-1},w^*)\\
\end{array}\right]+\tau\{d_{\ell \kappa}\}
\left[\begin{array}{c}
a_{h}(e_u^{n,0},w^*)\\
\vdots\\
a_{h}(e_u^{n,s-2},w^*)\\
a_{h}(e_u^{n,s-1},w^*)\\
\end{array}\right]
+\left[\begin{array}{c}
0\\
\vdots\\
0\\
(G^{n,s},w^*)\\
\end{array}\right],
\end{align}
where $e_u^{n,\ell}=u^{n,\ell}-u_h^{n,\ell}$ for $\ell=0,1,\cdots,s$. Also, we have $e_u ^{n,0}=e_u^{n}$ and $e_u ^{n,s}=e_u^{n+1,0}=e_u^{n+1}$.

Replacing $e_u$ in \eqref{any_order_e_u_1} with $\xi_u-\eta_u$, and similarly, we have $\xi_u ^{n,0}=\xi_u^{n}$ and $\xi_u ^{n,s}=\xi_u^{n+1,0}=\xi_u^{n+1}$ for $\xi$. The same representation applies to $\eta$ as well. Then, utilizing Lemma \ref{middle_1_2}, we have:
\begin{align}\label{any_order_xi_u_1}
\left[\begin{array}{c}
(\xi_u^{n,1},w^*)\\
\vdots\\
(\xi_u^{n,s-1},w^*)\\
(\xi_u^{n,s},w^*)\\
\end{array}\right] 
&=\left[\begin{array}{c}
(\eta_u^{n,1},w^*)\\
\vdots\\
(\eta_u^{n,s-1},w^*)\\
(\eta_u^{n,s},w^*)\\
\end{array}\right] +
\{c_{\ell \kappa}\}
\left[\begin{array}{c}
(\xi_u^{n,0},w^*)\\
\vdots\\
(\xi_u^{n,s-2},w^*)\\
(\xi_u^{n,s-1},w^*)\\
\end{array}\right]
-\{c_{\ell \kappa}\}
\left[\begin{array}{c}
(\eta_u^{n,0},w^*)\\
\vdots\\
(\eta_u^{n,s-2},w^*)\\
(\eta_u^{n,s-1},w^*)\\
\end{array}\right]\\
&\quad +\tau\{d_{\ell \kappa}\}
\left[\begin{array}{c}
a_{h}(\xi_u^{n,0},w^*)\\
\vdots\\
a_{h}(\xi_u^{n,s-2},w^*)\\
a_{h}(\xi_u^{n,s-1},w^*)\\
\end{array}\right]
+\left[\begin{array}{c}
0\\
\vdots\\
0\\
(G^{n,s},w^*)\\
\end{array}\right].\nonumber
\end{align}
For the projection error of the Lagrange operator $\eta^{n,\ell}_u=u^{n,\ell}-\mathcal{P}_hu^{n,\ell}$, we have the estimation \eqref{p_estimate_2}. This means that to estimate the error $e_u$, we only need to provide an accurate estimation for $\xi_u$, which involves estimating the convergence order for \eqref{any_order_xi_u_1}.

Adding the symbol $G^{n,\ell}=0$ for $\ell=1,\cdots,s-1$, we can then derive the error equation for each stage solution as:
\begin{align}\label{error_equation_6}
(\xi^{n,\ell+1}_u,w^*)&=\sum_{0\le\kappa\le l}c_{\ell,\kappa}(\xi^{n,\kappa}_u,w^*)+\tau\sum_{0\le\kappa\le \ell}d_{\ell,\kappa} a_h(\xi^{n,\kappa}_u,w^*)\\
&+(\eta^{n,\ell+1}_u-\sum_{0\le\kappa\le \ell}c_{\ell,\kappa}\eta^{n,\kappa}_u+G^{n,\ell+1},w^*).\nonumber
\end{align}

We will analyze the optimal convergence order of the error equation \eqref{any_order_xi_u_1} for RKSV(s,k) using the evolution of matrices. The overall evolution process will follow a similar analysis based on the stability of RKSV(s,k).
Before proceeding, let's introduce some explanations of symbols.

We define the operator $\mathbb{E}_\kappa$, where $\kappa\in \{0,1,\cdots,s\}$, with initial condition satisfying $\mathbb{E}_0 e_u^n=e_u^n$, and in other cases, it satisfies: 
\begin{align}\label{error_e_u_new}
    \mathbb{E}_\kappa e_u^n=\sum_{0\le \ell \le \kappa} \beta_{\kappa \ell}e_u^{n,\ell},\quad  \kappa\in \{1,\cdots,s\},
\end{align}
where the constant coefficients $\beta$ satisfy $\sum_{0\leq \ell \leq \kappa}\beta_{\kappa \ell}=0$. 
From \eqref{error_e_u_new}, the operator $\mathbb{E}_\kappa$ acting on $e_u^n$ can be seen as a linear representation of the stage error solution $e_u^{n,\ell}$.

It is worth noting that, so far, $\beta_{\kappa \ell}$ is still unknown. To further determine its value, we will utilize the relationship between inner product and bilinearity. For any $w\in \mathcal{U}_h$, the operator $\mathbb{E}_\kappa$ satisfies the following expression:
\begin{align}
    (\mathbb{E}_\kappa \xi_u^n,w^*)&=\tau a_h (\mathbb{E}_{\kappa-1} \xi_u^n,w^*) +  (\mathbb{E}_\kappa \eta_u^n,w^*)  , \quad \kappa=1,\cdots,s-1,\label{eta_trans_pro_1}\\
    (\mathbb{E}_s \xi_u^n,w^*)&=\tau a_h (\mathbb{E}_{s-1} \xi_u^n,w^*) +  (\mathbb{E}_s \eta_u^n,w^*)+(\tilde{G}^{n,s},w^*),\label{eta_trans_pro_2}
\end{align} 
where  the $\tilde{G}^{n,s}=\frac{1}{d_{s-1,s-1}}G^{n,s}$,  according to the definition \eqref{any_order_RK}, it is evident that $d_{s-1,s-1} $ is not equal to 0.

Due to the definition of $\tilde{G}^{n,\ell}=0$ for $\ell=1,2,\ldots,s-1$, equations \eqref{eta_trans_pro_1}-\eqref{eta_trans_pro_2} can be written in the following unified form:
\begin{align}\label{eta_trans_pro}
    (\mathbb{E}_\kappa \xi_u^n,w^*)=\tau a_h (\mathbb{E}_{\kappa-1} \xi_u^n,w^*) +  (\mathbb{E}_\kappa \eta_u^n,w^*)+(\tilde{G}^{n,\kappa},w^*)  , \quad \kappa=1,\cdots,s.
\end{align} 
\begin{remark}
   In equations \eqref{eta_trans_pro_1}-\eqref{eta_trans_pro_2}, the operator $\mathbb{E}_\kappa$ may not appear to simply transfer the temporal discretization term to the spatial discretization term, which distinguishes it from the operator $\mathbb{D}_\kappa$. At the same time, we have expressed the recursive relation for $\mathbb{E}_\kappa e_u^n$ as a recursive relation for $\mathbb{E}_\kappa \xi_u^n$, which can be observed by combining the second term on the right-hand side with the left-hand side, while utilizing the relevant properties of $a_h(\eta_u,v^*)$. This handling is aimed at facilitating the analysis of the term $\xi$ in the subsequent steps. It is worth noting that in the estimation of the last term $\mathbb{E}_s$, the term of temporal truncation error has been added.
\end{remark}

\begin{lemma}\label{D_eta_estimate}
    For the $\kappa \in\{1,\cdots,s\}$, we can get
   \begin{align}\label{D_estimate}
        \3bar\mathbb{E}_\kappa \xi_u^n \3bar \le C\lambda \3bar\mathbb{E}_{\kappa-1} \xi_u^n \3bar + C\tau h^{k+1} \|u\|_{k+2,\infty}+C \tau^{s+1}\|u\|_{s+2,\infty},
    \end{align}
    where $\lambda=\frac{\tau}{h}$ is the CFL number.
\end{lemma}
\begin{proof}
  Taking $w=\mathbb{E}_\kappa\xi_u^n$ in \eqref{eta_trans_pro} and combining it with the Cauchy-Schwarz inequality, we have:
      \begin{align}\label{D_estimate_1}
        \3bar\mathbb{E}_\kappa \xi_u^n \3bar \le C\lambda \3bar\mathbb{E}_{\kappa-1} \xi_u^n \3bar + \3bar\mathbb{E}_{\kappa} \eta_u^n\3bar \|u\|_{k+2,\infty} +\3bar \tilde{G}^{n,\kappa} \3bar \|u\|_{s+2,\infty},
    \end{align}
    By the \eqref{time_truncation}, we have
     \begin{align}\label{G_kappa_1}
        \3bar \tilde{G}^{n,\kappa} \3bar\le C\tau^{s+1} \|u\|_{s+2,\infty}
    \end{align}
    The $\mathbb{E}_{\kappa} \eta_u^n=\beta_{\kappa,\kappa} \eta_u^{n,\kappa}-\sum_{0\le\ell\le \kappa-1}\beta_{\ell,\kappa}\eta_u^{n,\ell}$, noticing $ \sum_{0\le\ell\le \kappa}\beta_{\ell,\kappa}=1  $, hence 
    \begin{align}\label{mathbb_D_eta}
        \3bar\mathbb{E}_{\kappa} \eta_u^n\3bar\le C\tau h^{k+1} \|u\|_{k+2,\infty}.
    \end{align}
    Combining the \eqref{mathbb_D_eta} and \eqref{G_kappa_1} into \eqref{D_estimate_1}, then the Lemma was proved.
\end{proof}

Based on the definition of $\mathbb{E}_\kappa$ given by \eqref{eta_trans_pro_1}-\eqref{eta_trans_pro_2}, we can derive the expression at $e_u^{n+1}$ as follows:
\begin{align}\label{eta_express_1}
    \beta_0e_u^{n+1}=\sum_{0\le \kappa\le s} \beta_\kappa \mathbb{E}_\kappa e_u^n.
\end{align}
Where $\beta_0>0$ is only used for scaling, let $\boldsymbol{\beta}=(\beta_0,\beta_1,\cdots,\beta_s)$.

Since $e_u^{n+1}\in \mathcal{H}_h$ and $\mathbb{E}_\kappa \xi_u^{n}\in \mathcal{U}_{h}$ for $\kappa\in \{0,1,\cdots,s\}$. Therefore, by taking the inner product of the left side of \eqref{eta_express_1} with $\beta_0 (\xi_u^{n+1})^*$, and the inner product of the right side of \eqref{eta_express_1} with $\sum_{0\le \kappa\le s}\beta _\kappa(\mathbb{E}_\kappa \xi_u^n)^*$, we have:
\begin{align*}
    \beta_0^2(e_u^{n+1},(\xi_u^{n+1})^*)=(\sum_{0\le p\le s}\beta _p\mathbb{E}_p e_u^n, \sum_{0\le q\le s}\beta _q(\mathbb{E}_q \xi_u^{n})^*).
\end{align*}

Using the definition of $\3bar \cdot \3bar$, the symmetry of inner product \eqref{inner_excharge_2}, and the error relationship \eqref{error_divided_2}, we obtain:
\begin{align*}
    \beta_0^2(\xi_u^{n+1},(\xi_u^{n+1})^*)&=\beta_0^2(\eta_u^{n+1},(\xi_u^{n+1})^*)+(\sum_{0\le p\le s}\beta _p\mathbb{E}_p \xi_u^n, \sum_{0\le q\le s}\beta _q(\mathbb{E}_q \xi_u^{n})^*)\\
    &\quad -(\sum_{0\le p\le s}\beta _p\mathbb{E}_p \eta_u^n, \sum_{0\le q\le s}\beta _q(\mathbb{E}_q \xi_u^{n})^*).
\end{align*} 
Further processing the above equation, we have the error equation:
\begin{align}
    \beta_0^2(\3bar\xi_u^{n+1}\3bar^2-\3bar\xi_u^{n} \3bar^2)&=\sum_{0\le p,q\le s} c_{pq}( \mathbb{E}_p \xi_u^n,(\mathbb{E}_q \xi_u^n)^*)+\sum_{0\le p,q\le s} d_{pq}( \mathbb{E}_p \eta_u^n,(\mathbb{E}_q \xi_u^n)^*)\nonumber\\
    &=\text{RHS}(\xi_u^n)\label{eta_express_2}
\end{align}
Where $c_{00}=0$, $c_{pq}=\beta_p\beta_q$ for $p+q>0$, and $d_{ss}=0, d_{pq}=-c_{pq}$ for $0\le p+q <2s$. 

According to equation \eqref{eta_express_2}, the optimal convergence order of RKSV(s,k) is determined by RHS$(\xi_h^n)$. 
In other words, once the estimation for RHS$(\xi_h^n)$ is provided, the result for the optimal convergence order corresponds accordingly.

Since the right-hand side of equation \eqref{eta_express_2} only depends on the inner product of the error in time stages, we attempt to use the equivalence relation \eqref{eta_trans_pro} to transform some inner products of the error in time stages into bilinear forms of spatial discretization errors. After using the equivalence relation $l$ times, where $l\geq0$, we obtain RHS$^{(l)}(\xi_h^n)$ as:
\begin{align}
    \text{RHS}^{(l)}(\xi_u^n)&=\sum_{0\le p,q\le s}c_{pq}^{(l)}(\mathbb{E}_p\xi^n_u,(\mathbb{E}_q\xi^n_u)^*)+\sum_{0\le p,q\le s}d_{pq}^{(l)}(\mathbb{E}_p\eta^n_u,(\mathbb{E}_q\xi^n_u)^*)\nonumber\\
    &+\sum_{0\le p,q\le s}g_{pq}^{(l)}(\tilde{G}^{n,p},(\mathbb{E}_q\xi^n_u)^*)+\sum_{0\le p,q\le s}h_{pq}^{(l)}\tau a_h(\mathbb{E}_p \xi^n_u,(\mathbb{E}_q\xi^n_u)^*).\label{RHS_define}
\end{align}
Here, we refer to $l$ as the number of transformations applied to RHS$(\xi_u^n)$. $c^{(l)}_{pq}$, $d^{(l)}_{pq}$, $g^{(l)}_{pq}$, $h^{(l)}_{pq}$ represent the elements in the $p$-th row and $q$-th column of matrices $\mathbb{C}^{(l)}$, $\mathbb{D}^{(l)}$, $\mathbb{G}^{(l)}$, and $\mathbb{H}^{(l)}$, respectively, where $p,q\in\{0,1,\cdots,s\}$.

Based on this, we will provide detailed explanations of the evolution process for RHS$^{(l)}(\xi_u^n)$.
\subsection{The evolution iteration of  RHS$^{(l)}(\xi_u^n)$}
According to the stability analysis in section \ref{stability_any_order_section}, the evolution iteration of RHS$(\xi_u^n)$ in the RKSV(s,k) follows a similar process as that of RHS$(u_h^n)$ in the stability analysis.

Firstly, we will provide the initial matrix for RHS$(\xi_u^n)$. The process described by the initial matrix corresponds to equation \eqref{eta_express_2}, so we have:
\begin{align}\label{RHS_initial_value}
   \mathbb{C}^{(0)}=\{c_{pq}\}, \mathbb{D}^{(0)}=\{d_{pq}\}, \mathbb{G}^{(0)}=\mathbb{O}, \mathbb{H}^{(0)}=\mathbb{O}.
\end{align}
Here, $\mathbb{C}^{(0)}$ represents the initial matrix of elements $c_{pq}$, $\mathbb{D}^{(0)}$ represents the initial matrix of elements $d_{pq}$, and $\mathbb{G}^{(0)}$ and $\mathbb{H}^{(0)}$ are both zero matrices $\mathbb{O}$, representing the initial matrices of elements $g_{pq}$ and $h_{pq}$, respectively. Additionally, $\mathbb{C}^{(0)}$, $\mathbb{D}^{(0)}$, $\mathbb{G}^{(0)}$, and $\mathbb{H}^{(0)}$ are all symmetric matrices of size $(s+1)\times(s+1)$.

The motivation behind the matrix transformations in \eqref{RHS_define} can be attributed to two reasons. Firstly, the equivalence relation \eqref{eta_trans_pro} transforms some inner products of the error in time stages into bilinear forms of spatial discretization errors, which means that elements in the matrix $\mathbb{C}^{l}$ are transformed through specific operations to matrices $\mathbb{C}^{l+1}$ and $\mathbb{H}^{l+1}$. This process also involves the temporal truncation error, represented by the matrix $\mathbb{G}^{l+1}$. Secondly, it is driven by fully exploiting the skew-symmetric properties of the approximated bilinear forms in spatial discretization, as described in equation \eqref{excharge_express}.

Next, to summarize the iterative evolution process of matrices, assuming we start from the $l\geq 1$ iteration, the matrices in RHS$^{(l)}(\xi_u^n)$ can be summarized as follows:
\begin{align*}
&\begin{aligned}
    \mathbb{C}^{(l)}=\left[\begin{array}{c:  c:  ccc}
        \mathbb{O} & \mathbb{O} & \mathbb{O} & \cdots & \mathbb{O} \\
        \hdashline \mathbb{O} & c_{l, l}^{(l)} & c_{l, l+1}^{(l)} & \cdots & c_{l, s}^{(l)} \\
        \hdashline \mathbb{O} & c_{l+1, l}^{(l)} & c_{l+1, l+2}^{(l)} & \cdots & c_{l+1, s}^{(l)} \\
        \vdots & \vdots & \vdots & \ddots & \vdots \\
        \mathbb{O} & c_{s, l}^{(l)} & c_{s, l+1}^{(l)} & \cdots & c_{s, s}^{(l)}
    \end{array}\right], 
    \end{aligned}
    \begin{aligned}
    \mathbb{D}^{(l)}= \left[\begin{array}{c:  c:  ccc}
        \star & \star & \star & \cdots & \star \\
        \hdashline \star & d_{l, l}^{(l)} & d_{l, l+1}^{(l)} & \cdots & d_{l, s}^{(l)} \\
        \hdashline \star & d_{l+1, l}^{(l)} & d_{l+1, l+1}^{(l)} & \cdots & d_{l+1, s}^{(l)} \\
        \vdots & \vdots & \vdots & \ddots & \vdots \\
        \star & d_{s, l}^{(l)} & d_{s, l+1}^{(l)} & \cdots & 0
    \end{array}\right], 
    \end{aligned}\\
     & \begin{aligned}
  \mathbb{G}^{(l)}=\left[\begin{array}{c:  c:  ccc}
        \star & \star & \star & \cdots & \star \\
        \hdashline \star & g_{l, l}^{(l)} & g_{l, l+1}^{(l)} & \cdots & g_{l, s}^{(l)} \\
        \hdashline \star & g_{l+1, l}^{(l)} & 0 & \cdots & 0 \\
        \vdots & \vdots & \vdots & \ddots & \vdots \\
        \star & g_{s, l}^{(l)} & 0 & \cdots & 0
    \end{array}\right],
\end{aligned}
\begin{aligned}
    \mathbb{H}^{(l)}=\left[\begin{array}{c:  c:  ccc}
        \star & \star & \star & \cdots & \star \\
        \hdashline \star & h_{l, l}^{(l)} & h_{l, l+1}^{(l)} & \cdots & h_{l,  s}^{(l)}  \\
        \hdashline \star &h_{l+1, l}^{(l)} & 0 & \cdots & 0 \\
        \vdots & \vdots & \vdots & \ddots & \vdots \\
        \star &h_{ s, l}^{(l)} & 0 & \cdots & 0
    \end{array}\right]. \\
\end{aligned}
\end{align*}
It is worth noting that although we represent the zero matrices in $\mathbb{C}^{(l)}$ as $\mathbb{O}$, the zero matrices $\mathbb{O}$ at different positions have different meanings. The zero matrix on the diagonal is denoted as $\mathbb{O}_{(l+1) \times (l+1)}$, the zero matrices from the second to the $s$-th in the first row are denoted as $\mathbb{O}_{(l+1) \times 1}$, and the zero matrices from the second to the $s$-th in the first column are denoted as $\mathbb{O}_{1 \times (l+1)}$. The same principle applies to the symbol $\star$, which represents the non-zero elements in matrices $\mathbb{D}^{(l)}$, $\mathbb{G}^{(l)}$, and $\mathbb{H}^{(l)}$.

Next, we will evolve RHS$^{(l)}(\xi_u^n)$ one step further to RHS$^{(l+1)}(\xi_u^n)$ using the equivalence relation \eqref{eta_trans_pro}. This involves transforming the elements in the $l+1$-th row of the matrix $\mathbb{C}^{(l)}$ into the remaining matrices. According to the symmetry of matrices, similar operations are performed for the $l+1$-th column of $\mathbb{C}^{(l)}$.

First, we check if $c_{l,l}^{(l)}$ is equal to zero. If it is zero, we can use the symmetry of the inner product $(\cdot, \cdot^*)$ to rewrite it as: 
\begin{align*}
    c_{l+1,l}^{(l)}(\mathbb{E}_{l+1} \xi_u^n, (\mathbb{E}_{l} \xi_u^n)^*)+c_{l,l+1}^{(l)}(\mathbb{E}_{l} \xi_u^n, (\mathbb{E}_{l+1} \xi_u^n)^*)=2c_{l+1,l}^{(l)}(\mathbb{E}_{l+1} \xi_u^n, (\mathbb{E}_{l} \xi_u^n)^*).
\end{align*}
Then, by fully exploiting the equivalence relation \eqref{eta_trans_pro}, we have: 
\begin{align}\label{diagonal_xi_define}
    2c_{l+1,l}^{(l)}(\mathbb{E}_{l+1} \xi_u^n, (\mathbb{E}_{l} \xi_u^n)^*)&=	2 c_{l+1,l}^{(l)}\tau a_h(\mathbb{E}_{l} \xi_u^n,(\mathbb{E}_{l} \xi_u^n)^*)+2 c_{l+1,l}^{(l)} (\mathbb{E}_{l+1} \eta_u^n,(\mathbb{E}_{l} \xi_u^n)^*)\nonumber\\
    &+ 2 c_{l+1,l}^{(l)}(\tilde{G}^{n,l+1}, (\mathbb{E}_{l} \xi_u^n)^*).
\end{align}
And for $l+1\leq p \leq s-1$, we have: 
\begin{align}
    &c_{p+1,l}^{(l)}( \mathbb{E}_{p+1} \xi_u^n, (\mathbb{E}_{l} \xi_u^n)^*)+c_{l,p+1}^{(l)}( \mathbb{E}_{l} \xi_u^n, (\mathbb{E}_{p+1} \xi_u^n)^*)
    +\delta_{p,l+1} c_{p,l+1}^{(l)}( \mathbb{E}_{p} \xi_u^n, (\mathbb{E}_{l+1} \xi_u^n)^*)\label{eta_involve_trans_2}\\
    &= \delta_{p,l+1} \left[c_{p,l+1}^{(l)}-\frac{2c_{p+1,l}^{(l)}}{\delta_{p,l+1}}\right](\mathbb{E}_p\xi_u^n,(\mathbb{E}_{l+1}\xi_u^n)^*) +2c_{p+1,l}^{(l)}left[\tau a_h(\mathbb{E}_p \xi_u^n,(\mathbb{E}_{l}\xi_u^n)^*)\nonumber\\
    &\quad+2c_{p+1,l}^{(l)}\tau a_h(\mathbb{E}_l \xi_u^n,(\mathbb{E}_{p}\xi_u^n)^*)
    + 2c_{p+1,l}^{(l)}(\mathbb{E}_{p+1} \eta_u^n,(\mathbb{E}_{l} \xi_u^n)^*)\nonumber+ 2 c_{p+1,l}^{(l)}(\tilde{G}^{n,p+1}, (\mathbb{E}_{l} \xi_u^n)^*).
\end{align}
Where when $p=l+1$, $\delta_{p,l+1}$ equals $1$, and in other cases, it equals $2$. Thus, we have completed the processing of elements in the $l$-th row, completing the $l+1$-th iteration.

Next, we will provide a systematic explanation of the evolution process of the matrix.
\begin{theorem}\label{matrix_xi_theorem_main}
                                            For $l \geq 0$, the matrix evolution of $\mathbb{C}$, $\mathbb{D}$, $\mathbb{G}$, and $\mathbb{H}$ will follow the following steps, assuming $q \leq p$:
\begin{itemize}
    \item Firstly, when $l=0$, the initial matrices are determined as follows: $\mathbb{C}^{(0)}=\{c_{pq}\}$, $\mathbb{D}^{(0)}=\{d_{pq}\}$, $\mathbb{G}^{(0)}=\mathbb{H}^{(0)}=\mathbb{O}$.
    \item Next, when $l \geq 1$, the matrices are updated. The updating of elements $c_{pq}^{(l)}$ in $\mathbb{C}^{(l)}$ satisfies the recursive formula:
    $$
c_{pq}^{(l)}=c_{qp}^{(l)}= 
\begin{cases}
    0, & q=l-1, \\ 
    c_{pq}^{(l-1)}-2 c_{p+1, q-1}^{(l-1)}, & p=l \text { and } q=l, \\ 
    c_{pq}^{(l-1)}-c_{p+1, q-1}^{(l-1)}, & l+1 \leq p \leq  s-1 \text { and } q=l, \\ 
    c_{pq}^{(l-1)}, & \text { otherwise }.
\end{cases}
$$
The updating of elements $d_{pq}^{(l)}$ in $\mathbb{D}^{(l)}$ satisfies the recursive formula:
$$
d_{pq}^{(l)}=d_{qp}^{(l)}=
\begin{cases}
    d_{pq}^{(l-1)}+2 c_{p q}^{(l-1)}, & l \leq p \leq s \text { and } q=l-1, \\ 
    d_{pq}^{(l-1)}, & \text { oterwise.}
\end{cases}
$$
The updating of elements $g_{pq}^{(l)}$ in $\mathbb{G}^{(l)}$ satisfies the recursive formula:
$$
g_{pq}^{(l)}=g_{qp}^{(l)}=
\begin{cases}
    2c_{pq}^{(l-1)}, & l\leq p \leq  s \text { and } q=l-1, \\ 
    g_{pq}^{(l-1)}, & \text { oterwise.}
\end{cases}
$$
The updating of elements $h_{pq}^{(l)}$ in $\mathbb{H}^{(l)}$ satisfies the recursive formula:
$$
h_{pq}^{(l)}=h_{qp}^{(l)}=
\begin{cases}
    2 c_{p, q}^{(l-1)}, & l \leq p \leq s \text { and } q=l-1, \\ 
    h_{pq}^{(l-1)}, & \text { oterwise}.
\end{cases}
$$
\item Thirdly, update $l$ to $l + 1$. Then, check if $c_{ll}^{(l)}$ equals 0. If it does, repeat the second step. If it does not equal 0, the matrix evolution concludes, and the termination index $\zeta$ is set to $l$.
\item Finally, we output the matrices $\mathbb{C}^{(\zeta)}$, $\mathbb{D}^{(\zeta)}$, $\mathbb{G}^{(\zeta)}$, $\mathbb{H}^{(\zeta)}$, and the termination index $\zeta$.
\end{itemize}
\end{theorem}
\subsection{The error estimate of RKSV(s,k)}
Following the analysis of the matrix evolution described above, we will now analyze RHS$^{(\zeta)}(\xi_u^n)$ to obtain the error estimate of RKSV(s,k).
\begin{align}\label{error_estimate_1}
    \alpha_0^2(\3bar\xi_u^{n+1}  \3bar^2-\3bar\xi_u^{n}  \3bar^2)&=\sum_{0\le p,q\le s}c_{pq}^{(\zeta)}(\mathbb{E}_p\xi^n_u,(\mathbb{E}_q\xi^n_u)^*)+\sum_{0\le p,q\le s}d_{pq}^{(\zeta)}(\mathbb{E}_p\eta^n_u,(\mathbb{E}_q\xi^n_u)^*)\nonumber\\
    &+\sum_{0\le p,q\le s}g_{pq}^{(\zeta)}(\tilde{G}^{n,p},(\mathbb{E}_q\xi^n_u)^*)+\sum_{0\le p,q\le s}h_{pq}^{(\zeta)}\tau a_h(\mathbb{E}_p \xi^n_u,(\mathbb{E}_q\xi^n_u)^*)\nonumber\\
    &=\mathbb{A}_1+\mathbb{A}_2+\mathbb{A}_3+\mathbb{A}_4.
\end{align}
The analysis of RHS$^{(\zeta)}(\xi_u^n)$ will be divided into four parts: $\mathbb{A}_1$, $\mathbb{A}_2$, $\mathbb{A}_3$, and $\mathbb{A}_4$. Next, we will proceed with the corresponding analysis of each part.

Before the formal analysis, we focus on the $(\kappa+1)$-order leading principal submatrix $\mathbb{H}^{\zeta}_\kappa$ of $\mathbb{H}^{(\zeta)}$. Let det$\mathbb{H}^{\zeta}_\kappa$ denote the value of the principal subdeterminant.
Let $\mathcal{H}=\{\kappa: \text{det} \mathbb{H}^{\zeta}_\kappa \leq 0, \text{ and } 0 \leq \kappa \leq \zeta-1\}$ be the set. Then, the indicator factor $\rho$ is given by:
\begin{align}\label{rho_define_1}
\rho= 
\begin{cases} min\{\gamma:  \gamma\in \mathcal{H}\} & \text{if} \quad \mathcal{H} \ne \varnothing,\\ 
    \zeta & \text { otherwise. }
\end{cases} 
\end{align}
Then, we divide the set $\pi=\{0,1,\ldots,s\}$ into three parts: $\pi_1=\{0,1,\ldots,\rho-1\}$, $\pi_2=\{\rho,\rho+1,\ldots,\zeta-1\}$, and $\pi_3=\{\zeta,\zeta+1,\ldots,s\}$. It is worth noting that when $\rho$ is 0, the set $\pi$ is divided into two parts $\pi_2$ and $\pi_3$. When $\rho$ is $\zeta$, the set $\pi$ is divided into two parts $\pi_1$ and $\pi_3$.
\begin{itemize}
\item Estimate the term of $ \mathbb{A}_1$.
 $$\mathbb{A}_1=\sum_{0\le p,q\le s}c_{pq}^{(\zeta)}(\mathbb{E}_p\xi^n_u,(\mathbb{E}_q\xi^n_u)^*).$$
 From the above expression, we can see that the estimate of $\mathbb{A}_1$ is an estimate of the error in the temporal discretization, utilizing the newly defined norm \eqref{norm_new}, and given that $c^{(\zeta)} _{\zeta\zeta}\neq 0$. Then, directly applying Cauchy-Schwarz inequality and Young's inequality, we have:
 	\begin{align}\label{first_term}
		\mathbb{A}_1&\le \frac{1}{2}c_{\zeta\zeta}^{(\zeta)}\3bar\mathbb{E}_{\zeta}\xi_u^n\3bar^2+C\sum_{l\in\pi_3\backslash  \{\zeta\}} \3bar\mathbb{E}_{l}\xi_u^n\3bar^2\nonumber\\
		&\le C [c_{\zeta\zeta}^{(\zeta)}+\lambda Q(\lambda)]\3bar\mathbb{E}_{\zeta}\xi_u^n\3bar^2+ C\tau^2 h^{2k+2}\|u\|^2_{k+2,\infty} +C \tau^{2s+2}\|u\|^2_{s+2,\infty}. 
	\end{align}
    where the constant $C > 0$ is independent of $\tau$, $h$, and $n$, and $Q_1(\lambda)$ is a polynomial in $\lambda$ with non-negative coefficients.
\item Estimate the term of $ \mathbb{A}_2$.
$$\mathbb{A}_2=\sum_{0\le p,q\le s}d_{pq}^{(\zeta)}(\mathbb{E}_p\eta^n_u,(\mathbb{E}_q\xi^n_u)^*).$$
Since $d_{00}^{(\zeta)}=d_{00}^{(0)}=-\beta_0^2$,  $\| \mathbb{E}_\kappa \eta _u^{n} \| \leq C\tau h^{k+1}\|u\|_{k+2,\infty}, \kappa\in \{1,\cdots,s\}$ and Young inequality, then there exists a constant $\epsilon_3$:
\begin{align*}
       \mathbb{A}_2&\le \sum_{0< p,q\le s}d_{pq}^{(\zeta)}(\mathbb{E}_p\eta^n_u,(\mathbb{E}_q\xi^n_u)^*)
       \le \frac{\epsilon_3}{4}\sum_{0< p\le s}\3bar \mathbb{E}_p\eta^n_u\3bar^2+\epsilon^{-1}_3\sum_{0< p\le s}\left(\sum_{0< q\le s}(d_{pq}^{(\zeta)})^2\right) \3bar\mathbb{E}_q\xi^n_u\3bar^2\\
       &\le \frac{C\epsilon_3}{4}\tau^2 h^{2k+2}\|u\|_{k+2,\infty} +C\lambda Q(\lambda)  \3bar\mathbb{E}_1\xi^n_u\3bar^2, 
   \end{align*}
     where the constant $C > 0$ is independent of $\tau$, $h$, and $n$.
     
 Utilizing Lemma \ref{D_eta_estimate}, we have:
   \begin{align}\label{xi_term_estimate_2}
        \mathbb{A}_2&\lesssim \epsilon \tau^2 h^{2k+2}\|u\|_{k+2,\infty}   +\tau \lambda Q_3(\lambda) \3bar\xi^n_u\3bar^2,
        \end{align}
 where $Q_3(\lambda)$ is a polynomial in $\lambda$ with non-negative coefficients.
\item Estimate the term of $ \mathbb{A}_3$.
\[
 \mathbb{A}_3=\sum_{0\le p,q\le s}g_{pq}^{(\zeta)}(\tilde{G}^{n,p},(\mathbb{E}_q\xi^n_u)^*) .
\]
Since $g_{00}^{(\zeta)}=g_{00}^{(0)}=0$ and $\tilde{G}^{n,p}=0, p\in\{0,1,\cdots,s-1\} $, we have
 \begin{align*}
        \mathbb{A}_3&= \sum_{1\le q\le s}g_{sq}^{(\zeta)}(\tilde{G}^{n,s},(\mathbb{E}_q\xi^n_u)^*) .
    \end{align*}
 Utilizing   \eqref{time_truncation}, Cauchy-Schwarz inequality and Young's inequality, then there exists a constant $\epsilon_4$:
\begin{align}
	\mathbb{A}_3&\le \frac{\epsilon_4}{4}\3bar \tilde{G}^{n,s}\3bar^2+ \epsilon_4^{-1} \sum_{1\le q\le s}(\sum_{1\le q\le s}g_{sq}^{(\zeta)})^2\3bar\mathbb{E}_q\xi^n_u\3bar^2 \nonumber\\
	&\le C \epsilon_4\tau^{2s+2} \|u\|_{s+2,\infty}+C  \lambda Q(\lambda)\3bar\mathbb{E}_1\xi^n_u\3bar^2.\label{third_term_estimate_1}
\end{align}
\item Estimate the term of $ \mathbb{A}_4$.
$$\mathbb{A}_4=\sum_{0\le p,q\le s}h_{pq}^{(\zeta)}\tau a_h(\mathbb{E}_p \xi^n_u,(\mathbb{E}_q\xi^n_u)^*).$$
For ease of analysis, we introduce symbols:
$$ \tilde{\mathbb{A}}(\pi_a,\pi_b)=\sum_{\substack{p\in \pi_a\\q\in \pi_b}}\tau (h_{pq}^{(\zeta)}-\epsilon \delta_{pq}) a_h(\mathbb{E}_p \xi^n_u,(\mathbb{E}_q\xi^n_u)^*)+ \sum_{\substack{p\in \pi_b\\q\in \pi_a}}\tau (h_{pq}^{(\zeta)}-\epsilon \delta_{pq}) a_h(\mathbb{E}_p \xi^n_u,(\mathbb{E}_q\xi^n_u)^*),$$
where $a,b\in \{1,2,3\}$ and $\delta_{pq}$ is the standard Kronecker delta symbol, $\epsilon>0$ is the minimum eigenvalue of the symmetric positive definite matrix $\{h_{pq}^{(\zeta)}\}_{p,q\in \pi_1}$.

The fourth term contains all terms in the form of $a_h(\cdot,\cdot^*)$, which can be decomposed as:
 \begin{align}\label{second_term}
        \mathbb{A}_4= \sum_{0\le p\le s}\epsilon \tau a_h&(\mathbb{E}_p\xi^n_u,(\mathbb{E}_p\xi^n_u)^*)+ 
        \frac{1}{2}\tilde{\mathbb{A}}(\pi_1,\pi_1)+ 
        \frac{1}{2}\tilde{\mathbb{A}}(\pi_2,\pi_2)\\
        &\quad + \frac{1}{2}\tilde{\mathbb{A}}(\pi_3,\pi_3)
         + \tilde{\mathbb{A}}(\pi_1,\pi_2)+ \tilde{\mathbb{A}}(\pi_1,\pi_3)+\tilde{\mathbb{A}}(\pi_2,\pi_3).\nonumber
    \end{align}
    For the first term on the right-hand side of $\mathbb{A}_4$, by the \eqref{excharge_express}, we have:
      \[
    \sum_{0\le p\le s}\epsilon \tau a_h(\mathbb{E}_p\xi^n_u,(\mathbb{E}_p\xi^n_u)^*)=-\epsilon\lambda \sum_{0\le p\le s}\|\mathbb{E}_p\xi^n_u\|_{\Gamma_h}^2\le 0.
    \]
    
    For the second term on the right-hand side, we note that $\{h_{pq}^{(\zeta)}-\epsilon \delta_{pq} \}_{p,q\in \pi_1}$ is a symmetric positive semi-definite matrix. Then, we can utilize the property of \eqref{positive_matrix_1}, and thus have:
    \[
    \frac{1}{2}\tilde{\mathbb{A}}(\pi_1,\pi_1)=\sum_{\substack{p\in \pi_1\\q\in \pi_1}}\tau (h_{pq}^{(\zeta)}-\epsilon \delta_{pq}) a_h(\mathbb{E}_p \xi^n_u,(\mathbb{E}_q\xi^n_u)^*)\le 0;
    \]

    According to \eqref{a_DG_estimate}, we obtain:
     \begin{align*}
        \frac{1}{2}\tilde{\mathbb{A}}(\pi_2,\pi_2)
        & =\sum_{\substack{p\in \pi_2\\q\in \pi_2}}\tau (d_{pq}^{(\zeta)}-\epsilon \delta_{pq}) a_h(\mathbb{E}_p \xi^n_u,(\mathbb{E}_q\xi^n_u)^*)\\
        &\le C\lambda \sum_{\substack{p\in \pi_2\\q\in \pi_2}} \|\mathbb{E}_p\xi^n_u\| \|\mathbb{E}_q\xi^n_u\| 
        \le C\lambda  \sum_{ p\in \pi_2}\|\mathbb{E}_p \xi^n_u\|^2.
    \end{align*}
    Similarly, we have the estimate:
     \begin{align*}
        \frac{1}{2}\tilde{\mathbb{A}}(\pi_3,\pi_3)
        \le C\lambda  \sum_{ p\in \pi_3}\|\mathbb{E}_p \xi^n_u\|^2.
    \end{align*}
    
According to \eqref{excharge_express}, we can express the fifth term as:
  \begin{align*}
        \tilde{\mathbb{A}}(\pi_1,\pi_2)=	-2\tau \sum_{\substack{p\in \pi_1\\q\in \pi_2}}  (h_{pq}^{(\zeta)}-\epsilon \delta_{pq}) \|[\mathbb{E}_p \xi^n_u]\|_{\Gamma_h}\|[\mathbb{E}_q \xi^n_u]\|_{\Gamma_h}.
    \end{align*}
 Then by the Young's inequality, we have the estimate 
 \begin{align*}
	\tilde{\mathbb{A}}(\pi_1,\pi_2) &\le \frac{1}{4}\gamma \tau \sum_{ p\in \pi_1}\|\mathbb{E}_p \xi^n_u\|^2_{\Gamma_h}+C\tau  \sum_{ p\in \pi_2}\|\mathbb{E}_p \xi^n_u\|^2_{\Gamma_h}\\
	&\le \frac{1}{4}\gamma\tau \sum_{ p\in \pi_1}\|\mathbb{E}_p\xi^n_u\|^2_{\Gamma_h} +C\lambda  \sum_{ p\in \pi_2}\|\mathbb{E}_p \xi^n_u\|^2.
\end{align*}
 Similarly, we have 
 \begin{align*}
	\tilde{\mathbb{A}}(\pi_1,\pi_3)\le \frac{1}{4}\gamma\tau \sum_{p\in \pi_1}\|\mathbb{E}_p\xi^n_u\|^2 +C\lambda  \sum_{ p\in \pi_3}\|\mathbb{E}_p \xi^n_u\|^2,\\
	\tilde{\mathbb{A}}(\pi_2,\pi_3)\le \frac{1}{4}\gamma\tau \sum_{p\in \pi_2}\|\mathbb{E}_p\xi^n_u\|^2 _{\Gamma_h}+C\lambda  \sum_{p\in \pi_3}\|\mathbb{E}_p\xi^n_u\|^2.
\end{align*}

    Based on the above analysis, we obtain the estimate of $\mathbb{A}_4$:
    	\begin{align*}
		\mathbb{A}_4 \le C\lambda  \sum_{ p\in \pi_2}\|\mathbb{E}_p \xi^n_u\|^2+C\lambda  \sum_{ p\in \pi_3}\|\mathbb{E}_p \xi^n_u\|^2
	\end{align*}
 Combine the Lemma\ref{D_eta_estimate}, we have 
 	\begin{align}\label{second_term_estimate_1}
	\mathbb{A}_4  \lesssim  \lambda Q_3(\lambda)\|\mathbb{E}_\rho \xi^n_u\|^2+\lambda Q_4(\lambda)\|\mathbb{E}_\zeta \xi^n_u\|^2+\tau^4 h^{2k} \|u\|^2_{k+2,\infty}+ \tau^{2s+4}h^{-2}\|u\|^2_{s+2,\infty}.
\end{align}
\end{itemize}

Next, we will present the conclusions regarding  the error estimate of RKSV(s,k). Use the \eqref{first_term},\eqref{xi_term_estimate_2},\eqref{third_term_estimate_1} and \eqref{second_term_estimate_1}, we have
\begin{align*}
	\3bar&\xi_u^{n+1}\3bar^2-\3bar\xi_u^{n} \3bar^2\lesssim [c_{\zeta\zeta}^{(\zeta)}+\lambda Q_1(\lambda)]\3bar\mathbb{E}_{\zeta}\xi_u^n\3bar^2+\lambda Q(\lambda) \3bar\mathbb{E}_1\xi^n_u\3bar^2
	+\lambda Q_4(\lambda)\|\mathbb{E}_\zeta \xi^n_u\|^2
	\\
	&+\tau^4 h^{2k} \|u\|^2_{k+2,\infty}+ \tau^{2s+4}h^{-2}\|u\|^2_{s+2,\infty}+\frac{\gamma}{4}\tau^2 h^{2k+2}\|u\|_{k+2,\infty}+ \gamma \tau^{2s+2} \|u\|_{s+2,\infty}.
\end{align*}
Taking $\gamma=\mathcal{O}(\tau^{-1})$ and let $\lambda<1$, we can get
\begin{align*}
	\3bar\xi_u^{n+1}\3bar^2-\3bar\xi_u^{n} \3bar^2\lesssim [c_{\zeta\zeta}^{(\zeta)}+\lambda^3+\lambda^{2\zeta}+\lambda^{2\rho+1}]\3bar\xi_u^n\3bar^2
	+\tau h^{2k+2}\|u\|_{k+2,\infty}+ \tau^{2s+1} \|u\|_{s+2,\infty}.
\end{align*}
Assume $\lambda^3=\mathcal{O}(\tau)$, then we can get 
\begin{align}
		\3bar\xi_u^{n+1}\3bar^2-\3bar\xi_u^{n} \3bar^2\lesssim [c_{\zeta\zeta}^{(\zeta)}+\lambda^{2\zeta}+\lambda^{2\rho+1}]\3bar\xi_u^n\3bar^2+\tau \3bar\xi_u^n\3bar^2
	+(\tau h^{2k+2}+\tau^{2s+1})\|u\|_{k+2,\infty}.
	\end{align}
\begin{theorem}\label{estimate_main_1}
  With the termination index $\zeta$ and the contribution index $\rho$ obtained by the above matrix transferring process for the $\text{RHS}(\xi_u^n)$, we have the following error estimate of  RKSV(s,k).
    \begin{enumerate}
        \item If $a_{\zeta \zeta}^{(\zeta)}<0$ , $\rho=\zeta$ and the CFL condition satisfies $\tau=\mathcal{O}(h)$ ,
        \item If $a_{\zeta \zeta}^{(\zeta)}<0$, $\rho<\zeta$ and the CFL condition satisfies $\lambda^{(2 \rho+1)}=\mathcal{\tau}$,
        \item If $a_{\zeta \zeta}^{(\zeta)}>0$, and the CFL condition satisfies  $\lambda^{(\gamma)}=\mathcal{\tau}$, $\gamma=\min (2 \zeta, 2 \rho+1)$.
    \end{enumerate}
     then the error estimate of the RKSV(s,k) is:
\begin{align}\label{main_option}
\|u^n - u_h^n\| \lesssim (h^{k+1} + \tau^s)\|u\|_{k+2,\infty} .
\end{align}
\end{theorem}
\begin{proof}
We will dividing it into three cases for analysis:
\begin{enumerate}
    \item  If $a_{\zeta \zeta}^{(\zeta)}<0$ and $\rho=\zeta$, then we partition the set $\pi=\{0,1,\cdots,s\}$ into two regions: $\pi_1=\{0,1,\cdots,\zeta-1\}$ and $\pi_3=\{\zeta,\zeta+1,\cdots,s\}$.
    Then, based on the analysis above, we conclude:
     	\begin{align*}
	\3bar\xi_u^{n+1}  \3bar^2-\3bar \xi_u^{n}  \3bar^2 \lesssim  (c_{\zeta\zeta}^{(\zeta)}+\lambda^{2\zeta})\3bar\xi_u^n\3bar^2
		+\tau \3bar\xi_u^n\3bar^2+\tau(h^{2k+2}+\tau^{2s})\|u\|_{k+2,\infty} 
	\end{align*}
    We have $c_{\zeta\zeta}^{(\zeta)}\leq 0$, and $\lambda$ is a sufficiently small number, then we can obtain:
    \begin{align}\label{result_1}
		\3bar\xi_u^{n+1}  \3bar^2-\3bar \xi_u^{n}  \3bar^2 \lesssim \tau \3bar\xi_u^n\3bar^2+\tau(h^{2k+2}+\tau^{2s})\|u\|_{k+2,\infty}  .
	\end{align}
    By the Gronwall inequality, the error estimate can be get :
	\begin{align}\label{result_1_finally}
		\3bar\xi_u^{n}  \3bar \lesssim (h^{k+1}+\tau^{s})\|u\|_{k+2,\infty}  .
	\end{align}
Noticing  $u-u_h=e_u=\xi_u-\eta_u$, thus we can get the  $e_u$ estimate by the triangle inequality and \eqref{result_1_finally}.
	\begin{align}\label{erroe_estimate_1}
		\|u^n-u_h^n\|\lesssim \3bar u^n-u_h^n\3bar\lesssim (h^{k+1}+\tau^{s})\|u\|_{k+2,\infty}  .
	\end{align} 
\item If $c_{\zeta \zeta}^{(\zeta)}<0$ and $\rho<\zeta$, then we partition the set $\pi=\{0,1,\cdots,s\}$ into three regions: $\pi_1=\{0,1,\cdots,\rho-1\}$, $\pi_2=\{\rho,\rho+1,\cdots,\zeta-1\}$, $\pi_3=\{\zeta,\zeta+1,\cdots,s\}$.

 We have $ c_{\zeta\zeta}^{(\zeta)}\le 0, $ and $\lambda$ is a sufficiently small number, then we can get 
 	\[
	\3bar\xi_u^{n+1}  \3bar^2-\3bar \xi_u^{n}  \3bar^2 \le( \tau+\lambda^{2\rho+1} )\3bar\xi_u^n\3bar^2+\tau(h^{2k+2}+\tau^{2s})\|u\|^2_{k+2,\infty}  . 
	\]	
    When the CFL condition $\lambda^{2\rho+1}=\mathcal{O}(\tau)$ is satisfied, we have:
     \begin{align}\label{result_2}
        \3bar\xi_u^{n+1}  \3bar^2-\3bar \xi_u^{n}  \3bar^2 \le C\tau \3bar\xi_u^n\3bar^2+C\tau(h^{2k+2}+\tau^{2s}) \|u\|^2_{k+2,\infty}  .
    \end{align}
    According to Gronwall's inequality , the error estimate is:
    \begin{align}\label{result_2_finally}
        \3bar\xi_u^{n}  \3bar \le C(h^{k+1}+\tau^{s}) \|u\|^2_{k+2,\infty}.
    \end{align}
    Note that $u-u_h=e_u=\xi_u-\eta_u$, thus we can obtain an estimate for $e_u$ through triangle inequality and \eqref{result_1_finally}.
	\begin{align}\label{erroe_estimate_2}
		\|u^n-u_h^n\|\lesssim \3bar u^n-u_h^n\3bar\lesssim (h^{k+1}+\tau^{s})\|u\|_{k+2,\infty}  .
	\end{align} 
	
    \item If  $c_{\zeta \zeta}^{(\zeta)}>0$ , then we divide three domain $\pi_1=\{0,1,\cdots,\rho-1\}$, $\pi_2=\{\rho,\rho+1,\cdots,\zeta-1\} $, $\pi_3=\{\zeta,\zeta+1,\cdots,s\} $ in the domain $\pi=\{0,1,\cdots,s\} $. 

Then, by the above analyze, we have:
	\[
\3bar\xi_u^{n+1}  \3bar^2-\3bar \xi_u^{n}  \3bar^2 \lesssim( \tau+\lambda^{2\rho+1}+\lambda^{2\zeta} )\3bar\xi_u^n\3bar^2+\tau(h^{2k+2}+\tau^{2s}) \|u\|_{k+2,\infty}  .
	\]	
We can take the CFL condition satisfy $max\{\lambda^{2\rho+1},\lambda^{2\zeta}\}=\mathcal{O}(\tau)$	, then we can get 
	\begin{align}\label{result_3}
		\3bar\xi_u^{n+1}  \3bar^2-\3bar \xi_u^{n}  \3bar^2 \lesssim \tau \3bar\xi_u^n\3bar^2+\tau(h^{2k+2}+\tau^{2s}) \|u\|_{k+2,\infty}  .
	\end{align}
By the Gronwall inequality, the error estimate can be get :
	\begin{align}\label{result_3_finally}
		\3bar\xi_u^{n}  \3bar \lesssim (h^{k+1}+\tau^{s})\|u\|_{k+2,\infty}  .
	\end{align}
Noticing  $u-u_h=e_u=\xi_u-\eta_u$, thus we can get the the  $e_u$ estimate by the triangle inequality and \eqref{result_3_finally}.
	\begin{align}\label{erroe_estimate_3}
		\|u^n-u_h^n\|\lesssim \3bar u^n-u_h^n\3bar\lesssim (h^{k+1}+\tau^{s})\|u\|_{k+2,\infty}  .
	\end{align} 
\end{enumerate}
\end{proof}

\begin{remark}\label{any_order_main_1}
If RKSV(s,k) exhibits monotonic stability and satisfies the CFL condition \(\tau = \mathcal{O}(h)\), and if RKSV(s,k) exhibits weak (\(\gamma\)) stability and the CFL condition satisfies \(\lambda^{\gamma} = \mathcal{O}(\tau)\), then the error estimate of the RKSV(s,k) is:
\begin{align*}
\|u^n - u_h^n\| \lesssim (h^{k+1} + \tau^s)\|u\|_{k,\infty}.
\end{align*}
There is a one-to-one correspondence between error estimation and the stability of numerical schemes, as evidenced by the requirement to satisfy distinct CFL conditions for each scheme in order to achieve optimal convergence order.
\end{remark}

Next, Below, we will present the key elements of stability and convergence in the RKSV(s,k) scheme. In the following table, we will list the crucial factor information for RKSV(1,k) to RKSV(12,k): the termination factor $\zeta$, the indication factor $\rho$, stability index $\gamma$, and CFL condition.
\begin{table}
\caption{Key Factors in RKSV(1,k) to RKSV(12,k)}
\label{example_1_12}
\centering
\begin{tabular}{ccccccc}
   \toprule
Schemes & $s$ &  $c_{\zeta \zeta}^{(\zeta)} $ & $\zeta $ & $\rho $ & $\gamma$ & CFL condition \\
 \midrule
RKSV$(1,k)$ & 1 &  1& 1&1 & 2&$\tau=\mathcal{O}(h^2)$ \\
RKSV$(2,k)$ & 2 & 1 &2 & 2& 4&$\tau=\mathcal{O}(h^{4/3})$ \\
RKSV$(3,k)$ & 3 & -3 &2 &2 & -&$\tau=\mathcal{O}(h)$ \\
RKSV$(4,k)$ & 4 &-8 &3 &2&5 &$\tau=\mathcal{O}(h^{5/4})$ \\
RKSV$(5,k)$ & 5 &40& 3& 3&6 &$\tau=\mathcal{O}(h^{6/5})$ \\
RKSV$(6,k)$ & 6 &180 &4 & 4&8 & $\tau=\mathcal{O}(h^{8/7})$\\
RKSV$(7,k)$ & 7 &  -1260& 4&4 &- & $\tau=\mathcal{O}(h)$\\
RKSV$(8,k)$ & 8 & -8064 &5 &4 &9 &$\tau=\mathcal{O}(h^{9/8})$ \\
RKSV$(9,k)$ & 9 & 72576 &5 &5 &10 & $\tau=\mathcal{O}(h^{10/9})$\\
RKSV$(10,k)$ & 10 &604800  & 6& 6&12 &$\tau=\mathcal{O}(h^{12/11})$ \\
RKSV$(11,k)$ & 11 &-6652800  & 6& 6&- &$\tau=\mathcal{O}(h)$ \\
RKSV$(12,k)$ & 12 &-68428800  & 7& 6&13 &$\tau=\mathcal{O}(h^{13/12})$ \\
 \bottomrule
\end{tabular}
\end{table}

The error estimate is $\mathcal{O}(\tau^s+h^{k+1})$ holds when the RKSV(s,k) satisfies the corresponding CFL condition, as shown in Table \ref{example_1_12}.

\subsection{Cases}
In this subsection, we will provide specific proofs for the error estimate of RKSV(s,k), where $s=1,2,3,4$. This involves demonstrating the evolution of numerical scheme matrices for different orders. Similar proofs for other orders can be provided in a analogous manner.	
\indent
\begin{case}
The error equation for RKSV(1,k) applied to hyperbolic equations can be expressed as:
\begin{align}	\label{Euler-SV_error_1_2}
    ( e_u^{n+1},w^*)&=( e_u^n,w^*)+\tau a_{h}(e_u^n,w^*)+(G^{n,1},w^*).
\end{align}	
\end{case}	
Then replacing  $\xi_u-\eta_u$ by $e_u$, we can rewrite \eqref{Euler-SV_error_1_2} as
\begin{align}\label{Euler_replace_equation}
(\xi^{n+1}_u,w^*)= (\xi_u^{n},w^*) +(\eta_u^{n+1}-\eta_u^n,w^*)+\tau a_h(\xi_u^n,w^*)+(G^{n,1},w^*).
\end{align}
Since $\mathbb{E}_0 e_u^n=e_u^n$  and the \eqref{D_estimate}, we have
$$(\mathbb{E}_1\xi_u^n,v^*)=\tau a_h(\mathbb{E}_0\xi_u^n, v^*)+(\mathbb{E}_1 \eta_u^n+\tilde{G}^{n,1},v^*),$$
where the $\tilde{G}^{n,1}=G^{n,1} $.
Then from  \eqref{Euler_replace_equation}, we obtain
\[
\mathbb{E}_1e_u^n= e_u^{n+1}-e_u^n.
\]
We conclude from the above formula that
$$ e_u^{n+1}=\mathbb{E}_1e_u^n+\mathbb{E}_0e_u^n,$$ 
hence that $ \boldsymbol{\beta}=(1,1)$ . 
Based on this, we can write down the initial matrix:
$$
\begin{aligned}
    &\mathbb{C}^{(0)}=\left[\begin{array}{cc}
        0 & 1 \\
        1 & 1 
    \end{array}\right], 
     &\mathbb{D}^{(0)}=\left[\begin{array}{cc}
        -1 & -1 \\
        -1 & 0
    \end{array}\right], 
    & \mathbb{G}^{(0)}=\mathbb{H}^{(0)}=\left[\begin{array}{cc}
        0 & 0\\
        0& 0 
    \end{array}\right]. 
\end{aligned}
$$
By the matrix transferring, we obtain
$$
\begin{aligned}
	&\mathbb{C}^{(1)}=\left[\begin{array}{cc}
		0 & 0 \\
		0 & 1 
	\end{array}\right], 
	&\mathbb{D}^{(1)}=\left[\begin{array}{cc}
		-1 & 0 \\
		0 & 0
	\end{array}\right], 
	&\quad \mathbb{G}^{(1)}=\left[\begin{array}{cc}
		0 & 2\\
		2& 0 
	\end{array}\right].
	& \mathbb{H}^{(1)}=\left[\begin{array}{cc}
		0 & 2\\
		2 & 0 
	\end{array}\right], 
\end{aligned}
$$
Since $c_{22}=1$ in $\mathbb{C}^{(1)}$, the evolution of the matrix stops, and finally, we obtain $\zeta=1$ and the indication factor $\rho=1$.
\begin{case}
The error equation for RKSV(2,k) applied to hyperbolic equations can be expressed as:
 \begin{align}
        ( e_u^{n,1},w^*)&=( e_u^n,w^*)+\tau a_{h}(e_u^n,w^*)
        \label{RK-SV_error_1_2},\\
        ( e_u^{n+1},w^*)
        &=\frac{1}{2}( e_u ^n+e_u^{n,1},w^*)
        +\frac{\tau}{2} a_{h}(e_u^{n,1},w^*)+(G^{n,2},w^*).\label{RK-SV_error_2_2}
    \end{align}
\end{case} 	
After replacing $\xi_u-\eta_u$ with $e_u$, we can transform \eqref{RK-SV_error_1_2}-\eqref{RK-SV_error_2_2} into:
\begin{align}
    ( \xi_u^{n,1},w^*)&=(\xi_u^n,w^*)
    +(\eta_u^{n,1}-\eta_u^n,w^*) +\tau a_{h}(\xi_u^n,w^*), \label{RK-LSV_error_1_1_2} \\
    ( \xi_u^{n+1},w^*)
    &=\frac{1}{2}( \xi_u ^n+\xi_u^{n,1},w^*)+\frac{1}{2}( 2\eta_u^{n+1}-\eta_u ^n-\eta_u^{n,1}, w^*)
    +\frac{\tau}{2} a_{h}(\xi_u^{n,1},w^*)+( G^{n,2}, w^*).\label{RK-LSV_error_1_2_2}
\end{align}
Noting that   $\mathbb{E}_0e_u^n=e_u^n$ and  the definition of $\mathbb{E}_1$ in \eqref{eta_trans_pro}, we have
$$(\mathbb{E}_1\xi_u^n,v^*)=\tau a_h(\mathbb{E}_0\xi_u^n, v^*)+(\mathbb{E}_1 \eta_u^n,v^*).$$
then by the \eqref{RK-LSV_error_1_1_2}
\[
\mathbb{E}_1e_u^n= e_u^{n,1}-e_u^n.
\]
According to  \eqref{eta_trans_pro}, $\mathbb{E}_2$ satisfies 
$$(\mathbb{E}_2\xi_u^n,v^*)=\tau a_h(\mathbb{E}_1\xi_u^n, v^*)+(\mathbb{E}_2 \eta_u^n+\tilde{G}^{n,2}, v^*),$$
where $\tilde{G}^{n,1}=2G^{n,1}$.

By \eqref{RK-LSV_error_1_1_2}-\eqref{RK-LSV_error_1_2_2} it is obvious that 
\[
\mathbb{E}_2e_u^n= 2e_u^{n+1}-2e_u^{n,1}.
\]
And through simple calculations between operators  $\mathbb{E}_0,\mathbb{E}_1, \mathbb{E}_2$, we obtain
\[
2e_u^{n+1}=2\mathbb{E}_0e_u^n+2\mathbb{E}_1e_u^n+\mathbb{E}_2e_u^n,
\]
and $ \boldsymbol{\beta}=(2,2,1)$. From this, we obtain the initial matrix:
$$
\begin{aligned}
    &\mathbb{C}^{(0)}=\left[\begin{array}{ccc}
        0 & 4&2 \\
        4 & 4&2\\
        2 & 2&1
    \end{array}\right],
     &\mathbb{D}^{(0)}=\left[\begin{array}{ccc}
        -4 & -4&-2 \\
        -4 & -4&-2\\
        -2 & -2&0
    \end{array}\right], 
    & \mathbb{G}^{(0)}=\mathbb{H}^{(0)}=\left[\begin{array}{ccc}
        0 & 0&0\\
        0 & 0&0\\
        0 & 0&0
    \end{array}\right]. 
\end{aligned}
$$
By the matrix transferring, we obtain
	\begin{align*}
	\begin{aligned}
	&\mathbb{C}^{(1)}=\left[\begin{array}{ccc}
			0 & 0&0 \\
			0 & 0&2\\
			0 & 2&1
		\end{array}\right], 
		&\mathbb{D}^{(1)}=\left[\begin{array}{ccc}
			-4 & 0&0 \\
			0 & -4&-2\\
			0 & -2&0
		\end{array}\right], \\
		&\quad \mathbb{G}^{(1)}=\left[\begin{array}{ccc}
			0 & 8&4 \\
			8 & 0&0\\
			4 & 0&0
		\end{array}\right],
		& \mathbb{H}^{(1)}=\left[\begin{array}{ccc}
			0 & 8&4 \\
			8 & 0&0\\
			4 & 0&0
		\end{array}\right]. 
	\end{aligned}
	\end{align*}
Since the value at $c_{11}^{(1)}=0$, the evolution of the matrix continues:
	\begin{align*}
	&\begin{aligned}
	&\mathbb{C}^{(2)}=\left[\begin{array}{ccc}
			0 & 0&0 \\
			0 & 0&0\\
			0 & 0&1
		\end{array}\right], 
	&\mathbb{D}^{(2)}=\left[\begin{array}{ccc}
			-4 & 0&0 \\
			0 & -4&0\\
			0 & 0&0
		\end{array}\right], 
  \end{aligned}\\
		&\quad \begin{aligned}
	\mathbb{G}^{(2)}=\left[\begin{array}{ccc}
			0 & 8&4 \\
			8 & 0&4\\
			4 & 4&0
		\end{array}\right].
		& \mathbb{H}^{(2)}=\left[\begin{array}{ccc}
				0 & 8&4 \\
			8 & 0&4\\
			4 & 4&0
		\end{array}\right], 
	\end{aligned}
	\end{align*}
Since the value  $c_{22}^{(2)}\ne0$, the evolution of the matrix stops, resulting in the termination factor $\zeta=2$ and the indication factor $\rho=2$.
\begin{case}
The error equation for RKSV(3,k) applied to hyperbolic equations can be expressed as:
 \begin{align}
        &( e_u^{n,l+1},w^*)=( e_u^{n,l},w^*)+\tau a_{h}(e_u^{n,l},w^*), \quad l =0,1,
        \label{RK3-SV_error_1}\\
        &( e_u^{n+1},w^*)
        =\frac{1}{3}( e_u ^n,w^*)+\frac{1}{2}( e_u ^{n,1},w^*)+\frac{1}{6}( e_u ^{n,2},w^*)
        +\frac{\tau}{6} a_{h}(e_u^{n,2},w^*)+(G^{n,3},w^*).\label{RK3-SV_error_2}
    \end{align}
\end{case}
Substitute  $\xi_u-\eta_u$ for $e_u$ in \eqref{RK3-SV_error_1}-\eqref{RK3-SV_error_2}, we obtain
\begin{align}
    &( \xi_u^{n,l+1},w^*)=(\xi_u^{n,l},w^*)
    +(\eta_u^{n,l+1}-\eta_u^{n,l},w^*) +\tau a_{h}(\xi_u^{n,l},w^*), \quad l =0,1, \label{RK3-LSV_error_1_1} \\
    &( \xi_u^{n+1},w^*)
    =\frac{1}{6}( 2\xi_u ^n+3\xi_u^{n,1}+\xi_u^{n,2},w^*)+\frac{1}{6}( 6\eta_u^{n+1}-2\eta_u ^n-3\eta_u^{n,1}-\eta_u^{n,2}, w^*)\label{RK3-LSV_error_1_2}
    \\
    &\quad\quad\quad\quad\quad +\frac{\tau}{6} a_{h}(\xi_u^{n,2},w^*)+( G^{n,3}, w^*).\nonumber
\end{align}	
Noting that   $\mathbb{E}_0e_u^n=e_u^n$ and  the definition of $\mathbb{E}_1$ in \eqref{eta_trans_pro}, we have
$$(\mathbb{E}_1\xi_u^n,v^*)=\tau a_h(\mathbb{E}_0\xi_u^n, v^*)+(\mathbb{E}_1 \eta_u^n,v^*).$$
then by the \eqref{RK3-LSV_error_1_1}, we can get
\[
\mathbb{E}_1e_u^n=e_u^{n,1}-e_u^{n}.
\]
By the definition of $\mathbb{E}_2$ in \eqref{eta_trans_pro}, we have
\[
(\mathbb{E}_2\xi_u^n,v^*)=\tau a_h(\mathbb{E}_1\xi_u^n, v^*)+(\mathbb{E}_2 \eta_u^n,v^*).
\]
then by the \eqref{RK3-LSV_error_1_1}, we can get
\[
\mathbb{E}_2e_u^n=e_u^{n,2}-2e_u^{n,1}+e_u^{n}.
\]	
According to  \eqref{eta_trans_pro}, $\mathbb{E}_3$ satisfies 
\[
(\mathbb{E}_3\xi_u^n,v^*)=\tau a_h(\mathbb{E}_2\xi_u^n, v^*)+(\mathbb{E}_3 \eta_u^n+\tilde{G}^{n,3},v^*),
\]
where $ \tilde{G}^{n,3}=6G^{n,3}$, we have 
\[
\mathbb{E}_3e_u^n=6e_u^{n,3}-3e_u^{n,2}-3e_u^{n},
\]
And  through simple calculations between operators $\mathbb{E}_0,\mathbb{E}_1, \mathbb{E}_2, \mathbb{E}_3$, we obtain
\[
6e_u^{n+1}=6\mathbb{E}_0e_u^n+6\mathbb{E}_1e_u^n+3\mathbb{E}_2e_u^n+\mathbb{E}_3e_u^n,
\]
and $ \boldsymbol{\beta}=(6,6,3,1)$. From this, we obtain the initial matrix:
$$
\begin{aligned}
    &\mathbb{C}^{(0)}=\left[\begin{array}{cccc}
        0 & 36&18&6 \\
        36 & 36&18&6\\
        18 & 18&9&3\\
        6 & 6&3&1
    \end{array}\right],
    &\mathbb{D}^{(0)}=\left[\begin{array}{cccc}
        -36 &- 36&-18&-6 \\
        -36 & -36&-18&-6\\
        -18 & -18&-9&-3\\
        -6 & -6&-3&0
    \end{array}\right], 
    \mathbb{G}^{(0)}= \mathbb{H}^{(0)}=\mathbb{O}_{5\times5},
\end{aligned}
$$
Therefore, according to the matrix evolution, we can obtain the matrix for $l=1$ as:
\begin{align*}
		&\begin{aligned}
			&\mathbb{C}^{(1)}=\left[\begin{array}{cccc}
				0 & 0&0&0 \\
				0 & 0&12&6\\
				0 & 12&9&3\\
				0& 6&3&1
			\end{array}\right], 
			&&&\mathbb{D}^{(1)}=\left[\begin{array}{cccc}
				-36 &0&0&0 \\
				0 & -36&-18&-6\\
				0 & -18&-9&-3\\
				0 & -6&-3&0
			\end{array}\right], 
		\end{aligned}\\
		&\begin{aligned}
			& \mathbb{G}^{(1)}=\left[\begin{array}{cccc}
				0 & 72&36&12 \\
				72& 0&0&0\\
			36 &0&0&0\\
				12 &0&0&0
			\end{array}\right],
			&&& \mathbb{H}^{(1)}=\left[\begin{array}{cccc}
				0 & 72&36&12 \\
				72& 0&0&0\\
				36 &0&0&0\\
				12 &0&0&0
			\end{array}\right]. 
		\end{aligned}	\end{align*}
Since $c_{11}^{(1)}=0$, the evolution of the matrix continues. Thus, for $l=2$, we have: 
	\begin{align*}
		&\begin{aligned}
			&\mathbb{C}^{(2)}=\left[\begin{array}{cccc}
				0 & 0&0&0 \\
				0 & 0&0&0\\
				0 & 0&-3&3\\
				0& 0&3&1
			\end{array}\right], 
			&& &\mathbb{D}^{(2)}=\left[\begin{array}{cccc}
				-36 &0&0&0 \\
				0 & -36&-6&0\\
				0 & -6&-9&-3\\
				0 & 0&-3&0
			\end{array}\right], 
		\end{aligned}\\
		&\begin{aligned}
			& \mathbb{G}^{(2)}=\left[\begin{array}{cccc}
			0 & 72&36&12 \\
			72& 0&24&12\\
			36 &24&0&0\\
			12 &12&0&0
			\end{array}\right],
			&& & \mathbb{H}^{(2)}=\left[\begin{array}{cccc}
			0 & 72&36&12 \\
			72& 0&24&12\\
			36 &24&0&0\\
			12 &12&0&0
			\end{array}\right]. 
		\end{aligned}
	\end{align*}
Since the value at $c_{22}^{(2)}$ is not equal to 0, the evolution of the matrix stops, resulting in the termination factor $\zeta=2$ and the indication factor $\rho=2$.
\begin{case}
The error equation for RKSV(4,k) applied to hyperbolic equations can be expressed as:
  \begin{align}
        ( e_u^{n,l+1},w^*)&=( e_u^{n,l},w^*)+\tau a_{h}(e_u^{n,l},w^*), \quad l =0,1,2,
        \label{RK4-SV_error_1}\\
        ( e_u^{n+1},w^*) &=\frac{3}{8}( e_u ^n,w^*)+\frac{1}{3}( e_u ^{n,1},w^*)+\frac{1}{4}( e_u ^{n,2},w^*)+\frac{1}{24}( e_u ^{n,3},w^*)\label{RK4-SV_error_2}\\
        &\quad \quad +\frac{\tau}{24} a_{h}(e_u^{n,3},w^*)+(G^{n,4},w^*).\nonumber
    \end{align}
\end{case}	
Substitute  $\xi_u-\eta_u$ for $e_u$ in \eqref{RK4-SV_error_1}-\eqref{RK4-SV_error_2}, we obtain
\begin{align}
    ( \xi_u^{n,l+1},w^*)&=(\xi_u^{n,l},w^*)
    +(\eta_u^{n,l+1}-\eta_u^{n,l},w^*) +\tau a_{h}(\xi_u^{n,l},w^*), \quad l =0,1,2, \label{RK4-LSV_error_1_1} \\
    ( \xi_u^{n+1},w^*)
    &=\frac{1}{24}( 9\xi_u ^n+8\xi_u^{n,1}+6\xi_u^{n,2}+\xi_u^{n,3},w^*)+\frac{\tau}{24} a_{h}(\xi_u^{n,3},w^*)+( G^{n,4}, w^*)\label{RK4-LSV_error_1_2}\\
    & \quad +\frac{1}{24}( 24\eta_u^{n+1}-9\eta_u ^n-8\eta_u^{n,1}-6\eta_u^{n,2}-\eta_u^{n,3}, w^*).\nonumber
\end{align}	
By the same analysis and $\mathbb{E}_0e_u^n=e_u^n$, we can get:
\begin{align*}
    \mathbb{E}_1e_u^n&=e_u^{n,1}-e_u^{n},
    \quad \mathbb{E}_2e_u^n=e_u^{n,2}-2e_u^{n,1}+e_u^{n},\\
    \mathbb{E}_3e_u^n&=e_u^{n,3}-3e_u^{n,2}+3e_u^{n,1}-e_u^{n}.
\end{align*}
According to  \eqref{eta_trans_pro}, $\mathbb{E}_4$ satisfies 
\[
(\mathbb{E}_4\xi_u^n,v^*)=\tau a_h(\mathbb{E}_3\xi_u^n, v^*)+(\mathbb{E}_4 \eta_u^n+\tilde{G}^{n,4},v^*),
\]
where $ \tilde{G}^{n,4}=24G^{n,4}$, we have 
\[
\mathbb{E}_4e_u^n=24e_u^{n+1}-4e_u^{n,3}-12e_u^{n,1}-8e_u^{n}.
\]
And  through simple calculations between operators $\mathbb{E}_0,\mathbb{E}_1, \mathbb{E}_2, \mathbb{E}_3,\mathbb{E}_4$ , we obtain
\[
24e_u^{n+1}=24\mathbb{E}_0e_u^n+24\mathbb{E}_1e_u^n+12\mathbb{E}_2e_u^n+4\mathbb{E}_3e_u^n+\mathbb{E}_4e_u^n,
\]
and $\boldsymbol{\beta}=(24,24,12,4,1)$.
From this, we obtain the initial matrix:
\begin{align*}
\begin{aligned}
    &\mathbb{C}^{(0)}=\left[\begin{array}{ccccc}
        0 & 576&288&96&24 \\
        576 & 576&288&96&24\\
        288 & 288&144&48&12\\
        96 & 96&48&16&4\\
        24 & 24&12&4&1
    \end{array}\right], 
     &\mathbb{D}^{(0)}=\left[\begin{array}{ccccc}
        -576 & -576&-288&-96&-24 \\
        -576 & -576&-288&-96&-24\\
        -288 & -288&-144&-48&-12\\
        -96 & -96&-48&-16&-4\\
        -24 &- 24&-12&-4&0
    \end{array}\right],
\end{aligned}
\end{align*}
$$\mathbb{G}^{(0)}=\mathbb{O}_{5\times 5},\quad \mathbb{H}^{(0)}=\mathbb{O}_{5\times 5}.$$
By the matrix transferring, we obtain
	\begin{align*}
		&\begin{aligned}
			&\mathbb{C}^{(1)}=\left[\begin{array}{ccccc}
				0 & 0&0&0&0 \\
				0 & 0&192&72&24\\
				0 & 192&144&48&12\\
				0 & 72&48&16&4\\
				0	& 24&12&4&1
			\end{array}\right], 
			&&   &\mathbb{D}^{(1)}=\left[\begin{array}{ccccc}
				-576 & 0&0&0&0 \\
				0& -576&-288&-96&-24\\
				0 & -288&-144&-48&-12\\
				0 & -96&-48&-16&-4\\
				0 &- 24&-12&-4&0
			\end{array}\right], 
		\end{aligned}\\
		&\begin{aligned}
			&\mathbb{G}^{(1)}=\left[\begin{array}{ccccc}
				0 & 1152&576&192&48 \\
				1152& 0&0&0&0\\
				576& 0&0&0&0\\
				192& 0&0&0&0\\
				48& 0&0&0&0
			\end{array}\right],
			&&  &\mathbb{H}^{(1)}=\left[\begin{array}{ccccc}
					0 & 1152&576&192&48 \\
				1152& 0&0&0&0\\
				576& 0&0&0&0\\
				192& 0&0&0&0\\
				48& 0&0&0&0
			\end{array}\right].
		\end{aligned}
	\end{align*}
Since $c_{11}^{(1)}=0$, the evolution of the matrix continues:
	\begin{align*}
		&\begin{aligned}
			&\mathbb{C}^{(2)}=\left[\begin{array}{ccccc}
				0 & 0&0&0&0 \\
				0 & 0&0&0&0\\
				0 & 0&0&24&12\\
				0 & 0&24&16&4\\
				0	& 0&12&4&1
			\end{array}\right], 
			&& &\mathbb{D}^{(2)}=\left[\begin{array}{ccccc}
				-576 & 0&0&0&0 \\
				0& -576&-96&-24&0\\
				0 & -96&-144&-48&-12\\
				0 & -24&-48&-16&-4\\
				0 &0&-12&-4&0
			\end{array}\right], 
		\end{aligned}\\
		&\begin{aligned}
			&\mathbb{G}^{(2)}=\left[\begin{array}{ccccc}
				0 & 1152&576&192&48 \\
			1152& 0&384&144&84\\
			576& 384&0&0&0\\
			192& 144&0&0&0\\
			48& 84&0&0&0
			\end{array}\right],
			&&  &\mathbb{H}^{(2)}=\left[\begin{array}{ccccc}
					0 & 1152&576&192&48 \\
				1152& 0&384&144&84\\
				576& 384&0&0&0\\
				192& 144&0&0&0\\
				48& 84&0&0&0
			\end{array}\right].
		\end{aligned}  
	\end{align*}
Since $c_{22}^{(2)}=0$, the evolution of the matrix continues:
	\begin{align*}
		&\begin{aligned}
			&\mathbb{C}^{(3)}=\left[\begin{array}{ccccc}
				0 & 0&0&0&0 \\
				0 & 0&0&0&0\\
				0 & 0&0&0&0\\
				0 & 0&0&-8&4\\
				0	& 0&0&4&1
			\end{array}\right], 
			&& &\mathbb{D}^{(3)}=\left[\begin{array}{ccccc}
				-576 & 0&0&0&0 \\
				0& -576&-96&-24&0\\
				0 & -96&-144&-24&0\\
				0 & -24&-24&-16&-4\\
				0 &0&0&-4&0
			\end{array}\right], 
		\end{aligned}\\
		&\begin{aligned}
			&\mathbb{G}^{(3)}=\left[\begin{array}{ccccc}
				0 & 1152&576&192&48 \\
			1152& 0&384&144&84\\
			576& 384&0&48&24\\
			192& 144&48&0&0\\
			48& 84&24&0&0
			\end{array}\right],
			&& &\mathbb{H}^{(3)}=\left[\begin{array}{ccccc}
				0 & 1152&576&192&48 \\
			1152& 0&384&144&84\\
			576& 384&0&48&24\\
			192& 144&48&0&0\\
			48& 84&24&0&0
			\end{array}\right],
		\end{aligned}
	\end{align*}	
Since $c_{33}^{(3)}$ is not equal to 0, the evolution of the matrix stops, resulting in the termination factor $\zeta=3$ and the indication factor $\rho=2$.

In the above four examples, through the evolution of matrices, we can ultimately obtain the values of $c_ {\zeta \zeta} ^ {(\zeta)} $, termination factor $\zeta$, and indication factor$\rho$. Therefore, using Theorem \ref {estimate_main_1}, we can directly obtain the optimal convergence order of the RKSV(s,k) under what CFL conditions it satisfies.

\section{Numerical results}
In this section, we numerically solve three examples, including a constant-coefficient linear hyperbolic equation, a degenerate variable-coefficient hyperbolic equation, and a two-dimensional linear hyperbolic equation. In the numerical experiments, for the time discretization, we will use RK(s) methods with $s=3,4,5,6$; for the spatial discretization, we will employ two spectral volume methods, RRSV and LSV. Additionally, for the case of degenerate variable coefficients, we will use the RSV method as a modification of RRSV, defined as follows: if $\alpha(x_{i-\frac{1}{2}}) \geq 0$ and $\alpha(x_{i+\frac{1}{2}}) \geq 0$, the control volume is constructed using the right Radau point; otherwise, the control volume is constructed using the left Radau point. If the chosen spatial order is $k$, then the fully discrete numerical schemes considered are the commonly used four types: RKSV(3,k), RKSV(4,k), RKSV(5,k), and RKSV(6,k). Different numerical examples will use different $k$ matched with corresponding time discretization.

To better present the numerical results, we define the following symbols: the number of spectral volume elements as $N$, the number of time discretization steps as $M$; $L_2$ norm error denotes the value of $\|e_h^n\|_0$, and $L_\infty$ norm error denotes the value of $\|e_h^n\|_{0, \infty}$.

\begin{example}\label{numerical_1}
We consider \eqref{linearEQ} with $\Omega=[0,2\pi]\times (0,1]$. The initial value and boundary value condition are respectively given by
$$u(x,0) = \text{sin}(x),\quad u(0,t) = u(2\pi,t). $$
The exact solution of this problem is $u(x,t) = \text{sin}(x-t).$
\end{example}

Next, we will use this example to demonstrate the conclusion from Theorem \ref{any_order_main_1}: the $L_2$ norm convergence order of  RKSV(s,k) schemes with $s=3,4$ is ${\mathcal O}(\tau^s+h^{k+1})$. 
To demonstrate this, we apply CFL conditions satisfying $\lambda=10^{-1}$  for the  RKSV(s,k) schemes with 
$s=3,4$. Consequently, the $L_2$ norm convergence order becomes ${\mathcal O}(h^{k+1})$.
 We repeat the same procedure for RKSV(s,k)  with $s=3,4$ but choose CFL conditions satisfying $\lambda=10^{-1}$. This also results in an $L_\infty$ norm convergence order of ${\mathcal O}(h^{k+1})$.

Table \ref{Table:case1_1} presents the $L_2$ norm error, $L_\infty$ norm error, and convergence order in space for the RK3-RRSV and RK3-LSV schemes. We observe that both fully discrete numerical schemes achieve the optimal convergence order of ${\mathcal O}(h^{k+1})$ for $L_2$ norm error, where $k=1,2,3$. This aligns with the theoretical proof.

\begin{table}[!t]
\caption{The $L_2$ and $L_\infty$ errors of RKSV(3,k) ($k=1,2,3$) at time $T=1$ (Example \ref{numerical_1})..}
\label{Table:case1_1}
\begin{tabular*}{\textwidth}{@{\extracolsep{\fill}}cccccccccc@{\extracolsep{\fill}}}
\hline
& & \multicolumn{2}{c}{RK3-RRSV}&   \multicolumn{2}{c}{RK3-RRSV}  &    \multicolumn{2}{c}{ RK3-LSV}   &  \multicolumn{2}{c}{RK3-LSV} \\
\cline{3-4}\cline{5-6}\cline{7-8}\cline{9-10}
$k$ & $N$ & $L_2$ & order &  $L_\infty$ & order & $L_2$ & order  &  $L_\infty$ & order \\
\hline
\multirow{5}{*}{1}
& 16    & 1.64e-02 &-   &  2.43e-02  & - & 2.28e-02 & -       &  3.14e-02  & -\\
&32  & 4.10e-03 &1.98  &  6.30e-03  &1.94 & 5.60e-03 &2.02 & 8.00e-03 &1.97  \\
&64   & 1.00e-03& 1.99  &  1.60e-03 &1.98 & 1.40e-03 & 2.00  & 2.00e-03 &1.99 \\
& 128  & 2.59e-04 & 1.99 &  4.01e-04 &1.99 & 3.46e-04 & 2.00&   5.01e-04 &1.99\\
\multirow{5}{*}{2}
&16  &5.20e-04 & -      & 9.93e-04  &-  & 8.23e-04& -       &  1.70e-03  & - \\
&32 &6.52e-05 &2.99 & 1.25e-04  &   2.98  &1.03e-04 &2.99 &  2.09e-04  &2.97 \\
&64  & 8.16e-06 & 2.99& 1.57e-05  & 2.99 & 1.28e-05 & 2.99  & 2.63e-05 &2.99\\
& 128  & 1.02e-06 & 2.99 & 1.96e-06  &2.99 & 1.61e-06 & 2.99&  3.29e-06 &2.99\\
\multirow{5}{*}{3}	
& 16   &1.25e-05 &-        & 3.14e-05& - & 2.06e-05 & -       &  5.29e-05 & -\\
&32	&	7.82e-07 &4.00	&1.99e-06  &3.97 & 1.28e-06 &4.00   &  3.32e-06  &3.99\\
& 64	&4.88e-08 & 4.00 & 1.25e-07 & 3.99 & 8.03e-08& 4.00   &  2.07e-07&3.99\\
& 128&	3.05e-09& 4.00     & 7.83e-09 & 3.99  &5.01e-09 & 4.00 & 1.29e-08 &3.99\\
\hline
\end{tabular*}
\end{table}

Table \ref{Table:case1_2} presents the $L_2$ norm error, $L_\infty$ norm error, and convergence order in space for the RK4-RRSV and RK4-LSV schemes. We observe that both fully discrete numerical schemes achieve the optimal convergence order of ${\mathcal O}(h^{k+1})$ for $L_2$ norm error, where $k=2,3,4$. This matches the theoretical proof.
\begin{table}[!t]
\caption{The $L_2$ and $L_\infty$ errors of RKSV(4,k) ($k=2,3,4$) at time $T=1$ (Example \ref{numerical_1}). .}
\label{Table:case1_2}
\begin{tabular*}{\textwidth}{@{\extracolsep{\fill}}cccccccccc@{\extracolsep{\fill}}}
\hline
    & & \multicolumn{2}{c}{RK4-RRSV}&   \multicolumn{2}{c}{RK4-RRSV}  &    \multicolumn{2}{c}{ RK4-LSV}   &  \multicolumn{2}{c}{RK4-LSV} \\
    \cline{3-4}\cline{5-6}\cline{7-8}\cline{9-10}
    $k$ & $N$ & $L_2$ & order &  $L_\infty$ & order & $L_2$ & order  &  $L_\infty$ & order \\
\hline
    \multirow{5}{*}{2}
    & 16    & 5.20e-04 &-   &  9.92e-04 & - & 8.03e-04 & -       &  1.70e-03  & -\\
    &32  & 6.52e-05 &2.99  & 1.25e-04  &2.98& 1.03e-04 &2.99 & 2.09e-04 &2.97 \\
    &64   & 8.15e-06& 2.99 &  1.56e-05 &2.99 & 1.28e-05 & 2.99  & 2.63e-05 &2.99 \\
    & 128  & 1.02e-06 & 2.99 &  1.96e-06 &2.99 & 1.61e-06 & 2.99&   3.29e-06 &2.99\\
    \multirow{5}{*}{3}
    &16  &1.21e-05 & -      & 3.01e-05  &-  & 2.03e-05& -       &  5.14e-05  & - \\
    &32 &7.56e-07 &4.00 & 1.91e-06  &   3.97  &1.26e-06 &4.00 &  3.22e-06  &3.99 \\
    &64  & 4.72e-08 & 4.00& 1.20e-07  &3.99 & 7.93e-08 & 4.00  & 2.01e-07 &3.99\\
    & 128  & 2.95e-09 & 4.00 & 7.52e-09  &3.99 & 4.95e-09 &4.00&  1.25e-08 &3.99\\
    \multirow{5}{*}{4}	
    & 16   &1.90e-07 &-        & 6.73e-07& - & 3.71e-07 & -       &  1.22e-06 & -\\
    &32	&	5.95e-09 &4.99	&2.12e-08  &4.98& 1.16e-08 &4.99  &  3.86e-08  &4.98\\
    & 64	&1.86e-10 & 4.99 & 6.64e-10& 4.99 & 3.64e-10& 4.99  &  1.21e-09&4.99\\
    & 128&	5.81e-12& 4.99   & 2.07e-11 & 4.99  &1.13e-11 & 4.99 & 3.78e-11 &4.99\\
    \hline
\end{tabular*}
\end{table}

Tables \ref{Table:case1_1} and \ref{Table:case1_2} also show the $L_\infty$ norm error and corresponding convergence order for the aforementioned SV methods (RK3-RRSV, RK3-LSV with $k=1, 2, 3$; RK4-RRSV, RK4-LSV with $k=2, 3, 4$). The $L_\infty$ error has not been theoretically investigated in this chapter. As shown in Tables \ref{Table:case1_1} and \ref{Table:case1_2}, we observe numerical convergence behavior similar to the $L_2$ norm error.

\begin{example}\label{numerical_2} We consider the equation
$$ u_t+(sin(x)u)_x=g(x,t), (x,t)\in [0,2\pi]\times (0,0.1],$$ 
with the initial value $u(x,0)=sin(x)$ and boundary value conditions $u(0,t)=u(2\pi,t)$. The problem admits the exact solution  $u(x,t)=e^{sin(x-t)}$.  Note that this example is different from Example \ref{numerical_1} in which the coefficient $\alpha(x)=sin(x)$  is a degenerate variable.
\end{example}

We numerically solve this problem over the following non-uniform meshes. 
We  divide the interval $[0,2\pi]$ into $N$ subintervals with $N=32,\cdots,256$, which are obtained by randomly and independently perturbing nodes of a uniform mesh  up to some percentages. To be more precise, we let 
$$x_i=\frac{2\pi i}{N}+\frac{1}{100N}\sin(\frac{i\pi}{N}){\rm randn}(),\quad 0\leq i\leq N,$$
where randn() returns a uniformly distributed random number in $(0,1)$.

Table \ref{Table:case2_1} describes the numerical results for the $L_2$ norm and $L_\infty$ norm measurements of the RKSV(5,k) schemes with $k=3,4,5$. It is worth noting that in both tables, we only provide errors and corresponding orders with respect to the spatial grid size $h$. To minimize errors introduced by time discretization and ensure satisfactory fulfillment of CFL conditions, we choose a CFL constant of $\lambda=10^{-3}$. 

From Table \ref{Table:case2_1}, we observe the optimal convergence order of ${\mathcal O}(h^{k+1})$ for both $L_2$ and $L_\infty$ norm errors, demonstrating our theoretical findings in Theorem \ref{any_order_main_1} and suggesting that similar theoretical analysis can be established for $L_\infty$ norm errors.
\begin{table}[!t]\label{Table:case2_1}
\caption{The $L_2$ and $L_\infty$ errors of RKSV(5,k) ($k=3,4,5$) at time $T=1$ (Example \ref{numerical_2}).}
\begin{tabular*}{\textwidth}{@{\extracolsep{\fill}}cccccccccc@{\extracolsep{\fill}}}
\toprule
& & \multicolumn{2}{c}{RK5-RSV}&   \multicolumn{2}{c}{RK5-RSV}  &    \multicolumn{2}{c}{ RK5-LSV}   &  \multicolumn{2}{c}{RK5-LSV} \\
\cline{3-4}\cline{5-6}\cline{7-8}\cline{9-10}
    $k$ & $N$ & $L_2$ & order &  $L_\infty$ & order & $L_2$ & order  &  $L_\infty$ & order \\
    \midrule
    \multirow{5}{*}{3}
    & 32   & 1.98e-01 &-    &  1.77e-01  & -    & 1.98e-01 &-    &  1.77e-01  & - \\
    &64  & 1.23e-02   &4.01 &  1.00e-02  &4.15  & 1.23e-02   &4.01 &  1.00e-02  &4.15  \\
    &128   & 7.68e-04 &4.00 &  6.20e-04  &4.00  & 7.68e-04 &4.00 &  6.20e-04  &4.00 \\
    & 256  & 4.80e-05 &4.00 &  3.87e-05  &4.00 & 4.80e-05 &4.00 &  3.87e-05  &4.00\\
    \multirow{5}{*}{4}
    &32  &3.86e-02   & -      & 3.19e-02  &-  &3.86e-02   & -      & 3.19e-02  &- \\
    &64 &1.20e-03    &5.00 & 9.74e-04  &  5.03  &1.20e-03    &5.00 & 9.74e-04  &  5.03 \\
    &128  & 3.77e-05 & 5.00& 3.04e-05  & 5.00  & 3.77e-05 & 5.00& 3.04e-05  & 5.00\\
    & 256  & 1.17e-06 & 5.00 & 9.51e-07  &5.00 & 1.17e-06 & 5.00 & 9.51e-07  &5.00\\
    \multirow{5}{*}{5}	
    & 32   &7.60e-03 &-      & 6.10e-03& - &7.60e-03 &-      & 6.10e-03& -\\
    &64	&	1.18e-04 &6.00	&9.56e-05  &6.00 &	1.18e-04 &6.00	&9.56e-05  &6.00\\
    & 128	&1.85e-06 & 6.00 & 1.49e-06 & 5.99 	&1.85e-06 & 6.00 & 1.49e-06 & 5.99\\
    & 256 &	2.89e-08& 6.00   & 2.19e-08 & 6.00 &	2.89e-08& 6.00   & 2.19e-08 & 6.00\\
    \bottomrule
\end{tabular*}
\end{table}

\begin{example}\label{numerical_3}
Consider a two-dimensional hyperbolic equation with periodic boundary conditions:
    \begin{align*}
			&u_t+u_x+u_y=0, \quad (x, y, t) \in[0, 1]\times[0, 1]\times (0, 0.1], \\
			&u(x, y, 0) = sin(x+y).
		\end{align*}
  The exact solution is \( u(x, y, t) = \sin(x+y-2t). \)
\end{example}

The current theory does not cover two-dimensional linear hyperbolic equations. However, we can evaluate the applicability of RKSV(6,k) (with $k=2,3,4$) in solving two-dimensional hyperbolic equations using numerical methods. We partition the interval $[0, 1]\times[0, 1]$ into $N\times N$ subintervals, where $N={16, 32, 64, 128}$.

Table \ref{2_Table:case_3_1} displays the $L_2$ and $L_\infty$ norm errors and their corresponding convergence orders obtained using RKSV(6,s) schemes at $T=0.1$, where $k=2,3,4$. We only present errors and corresponding orders related to the spatial grid size $h$. To control the influence of time discretization on numerical results, we choose a time step $\tau=0.002*h^{5}$. From Table \ref{2_Table:case_3_1}, it can be observed that the optimal convergence order of ${\mathcal O}(h^{k+1})$ has been achieved for solving the linear two-dimensional hyperbolic equation.

\begin{table}[!t]\label{Table:case2_1}
\caption{The $L_2$ and $L_\infty$ errors of RKSV(6,k) ($k=2,3,4$) at time $T=0.1$ (Example \ref{numerical_3}).}
\label{2_Table:case_3_1}
\centering
\begin{tabular*}{\textwidth}{@{\extracolsep{\fill}}cccccccccc@{\extracolsep{\fill}}}
\toprule
&  & \multicolumn{2}{c}{RK6-RRSV}&   \multicolumn{2}{c}{RK6-RRSV}  &    \multicolumn{2}{c}{ RK6-LSV}   &  \multicolumn{2}{c}{RK6-LSV} \\
\cline{3-4}\cline{5-6}\cline{7-8}\cline{9-10}
    $k$ & $N$ & $L_2$ & order &  $L_\infty$ & order & $L_2$ & order  &  $L_\infty$ & order \\
    \midrule
    \multirow{4}{*}{2}
    & 16    &1.80e-03   &-   & 6.40e-03   & - &1.60e-03   & -  & 3.90e-03   & -\\
    &32  & 2.21e-04  & 2.99  & 8.58e04   & 2.90&1.96e-04   &2.99 & 5.07e-04  & 2.95 \\
    &64   & 2.75e-05  & 3.00  &1.09e-04    &2.97 & 2.46e-05  & 2.99&6.36e-05   & 2.99 \\
    & 128 &3.26e-06   & 3.08  & 1.29e-05   & 3.08& 3.04e-06  & 3.01& 7.74e-06  &3.03  \\
    \multirow{4}{*}{3}	
    & 16    &9.07e-05   &-   &2.61e-04    & - &8.09e-05   & -  & 2.36e-04   & -\\
    &32	&5.69e-06   &3.99   & 167e-05   &3.96 &5.07e-06   & 3.99& 1.48e-05  & 3.99 \\
    & 64	& 3.56e-07  &3.99   &1.04e-06    &3.99 &3.16e-07   &4.00 &9.23e-07   &4.00  \\
    & 128   & 2.24e-08  & 3.98   &6.77e-08 &3.95   &1.97e-08 &  4.00 & 5.32e-08 &4.11  \\
    \multirow{4}{*}{4}
    &16    & 2.19e-06  &-   &4.74e-06    & - & 1.50e-06  & -  &  3.96e-06  & -\\
    &32 & 6.88e-08  & 4.99  &1.50e-07    &4.97 & 4.71e-08  &4.99 & 1.31e-07  & 4.92 \\
    &64  &2.15e-09   &4.99   &4.72e-09    & 4.99&1.47e-09   &5.00 &4.14e-09   & 4.98 \\
    & 128  &6.96e-11   &4.95   &1.54e-10    &4.93 &4.36e-11   & 5.07&1.21e-10   & 5.07 \\
    \bottomrule
\end{tabular*}
\end{table}

In the numerical experiments section, we focus exclusively on instances where the temporal discretization method is RK(s) with 
$s$ values of 3, 4, 5, and 6. This design choice is motivated by two considerations: firstly, when temporal accuracy is further enhanced without a corresponding increase in spatial accuracy, the overall accuracy of the fully discretized numerical scheme does not significantly improve. Secondly, from the perspective of computational efficiency, it is well-known that higher accuracy entails greater computational cost. However, this additional cost does not necessarily translate well into improved numerical accuracy. As the order of spatial discretization increases, we observe that the magnitude of the $L_2$ error reaches e-08. Such accuracy is already more than sufficient for general computational requirements.

\section{Conclusion}
In this work, we have investigated the stability and optimal order convergence of fully discrete RKSV(s,k) numerical schemes for linear hyperbolic conservation laws within a unified framework. Leveraging the evolution of matrices within this framework, we have ultimately demonstrated that when the stability of the fully discrete RKSV(s,k) scheme exhibits monotonic stability, under the CFL condition $\tau={\mathcal O}(h)$, the error achieves optimal order ${\mathcal O}(\tau^s+h^{k+1})$; and when the stability of the fully discrete RKSV(s,k) scheme manifests weak ($\gamma$) stability, under the CFL condition satisfying $\lambda^\gamma={\mathcal O}(\tau)$, the error also attains optimal order ${\mathcal O}(\tau^s+h^{k+1})$.
Simultaneously, through numerical experiments, we have validated that the SV method achieves optimal convergence rates in both $L_2$ and $L_\infty$ norms and excels in handling high-dimensional and variable-coefficient problems.

\section*{Acknowledgments}
The research was supported in part by the National Key R$\&$D Program of China (2022ZD0117805), by the National Natural Science Foundation of China under grant 12071496, and by the Guangdong Province grant  2023A1515012079.

\bibliography{mybib}
\bibliographystyle{IEEEtran}






\end{document}